\documentclass[11pt,reqno,intlimits]{amsart}%{article}
\usepackage{amsfonts,amssymb,amsmath,amsthm}
\usepackage{url}
\usepackage{graphicx}
\usepackage[colorlinks=true,urlcolor=blue,citecolor=red,linkcolor=blue,linktocpage,pdfpagelabels,bookmarksnumbered,bookmarksopen]{hyperref}

\newtheorem{theorem}{Theorem}[section]

\newtheorem{proposition}[theorem]{Proposition}
\newtheorem{lemma}[theorem]{Lemma}
\newtheorem{corollary}[theorem]{Corollary}
\newtheorem{remark}[theorem]{Remark}
\newcommand{\sezione}[1]{\section{#1}\setcounter{equation}{0}}
\newcommand{\nor}{\Arrowvert}

\def\R{{\rm I\mskip -3.5mu R}}
\def\N{{\rm I\mskip -3.5mu N}}

\def\e{\varepsilon}
\def\di12{\mathcal{D}^{1,2}(\R^n)}

\def\na{\nabla}
\def\d{\delta}
\def\g{\gamma}
\def\D{\Delta}
\def\l{{\lambda}}
\def\L{{\Lambda}}
\def\a{{\alpha}}
\def\b{{\beta}}

\def\0l{_{0,\l}}
\def\1l{_{1,\l}}
\def\2l{_{2,\l}}
\def\3l{_{3,\l}}
\def\4l{_{4,\l}}

\def\de{\partial}

\def\Om{\Omega}

\def\h{\mathcal{H}}
\def\mg{\mathcal{G}}

\begin{document}
\title[On the Hardy-Sobolev]{On the Hardy-Sobolev equation}

\thanks{The first author is partially supported by the Australian Research Council. The second author is supported by Gruppo Nazionale per
l'Analisi Matematica, la Probabilit\'a e le loro Applicazioni (GNAMPA)
of the Istituto Nazionale di Alta Matematica (INdAM). The last two authors are supported by PRIN-2012-grant ``Variational and perturbative aspects of nonlinear differential problems''.}
\author[Dancer]{E. N. Dancer}
\address{School of Mathematics and Statistics F07 University of Sydney NSW 2006 Australia and Department of Mathematics, University of Swansea, Swansea, UK.}
\email{norman.dancer@sydney.edu.au }
\author[Gladiali]{F.  Gladiali}
\address{Dipartimento Polcoming-Matematica e Fisica, Universit\`a  di Sassari  - Via Piandanna 4, 07100 Sassari - Italy.}
\email{fgladiali@uniss.it}
\author[Grossi]{M. Grossi}
\address{Dipartimento di Matematica, Sapienza Universit\`a di Roma, P.le A. Moro 2 - 00185 Roma- Italy.}
\email{massimo.grossi@uniroma1.it}

\begin{abstract}
In this paper we study the problem
\begin{equation}\tag{1}
\left\{\begin{array}{ll}
-\Delta u-\frac {\l}{|x|^2}u=u^p & \hbox{ in }\Omega\\
u\geq 0& \hbox{ in }\Omega
%u\in H^1_0(\Omega),
\end{array}\right.
\end{equation}
where $\Omega=\R^N$ or $\Omega=B_1$, $N\geq 3$, $p>1$ and $\l< \frac{(N-2)^2}4$. Using a suitable map we transform the problem $(1)$
into a another one without the singularity $\frac1{|x|^2}$. Then we obtain some bifurcation results from the radial solutions corresponding to some explicit values of $\l$.  
\end{abstract}
\maketitle
\sezione{Introduction, statement of the main results and idea of the proofs.}\label{s0}
In this paper we consider the following problem
\begin{equation}\label{first}
\left\{\begin{array}{ll}
-\Delta u-\frac {\l}{|x|^2}u=u^p & \hbox{ in }\Omega\\
u\geq 0& \hbox{ in }\Omega
\end{array}\right.
\end{equation}
where $ \Omega\subseteq \R^N$, $N\geq 3$, $p>1$ and $\l< \frac{(N-2)^2}4$.
%$C(\l)=N(N-2)\left(1-\frac {4\l}{(N-2)^2}\right)$.\\
We will focus on the case where either $\Omega=\R^N$ or $\Omega=B_1$ and $p$ suitably chosen.  By solutions we meen weak solutions, so we will ask that $u\in D^{1,2}\left(\R^N\right)$ where $D^{1,2}\left(\R^N\right)=\{u\in L^{2^*}(\R^N)\hbox{ such that } |\na u| \in L^2(\R^N)\}$ in the first case, or $u\in H^1_0(B_1)$ in the case of the ball.
These problems were very studied in the pasts years using both variational or moving plane methods or the finite dimensional reduction of Lyapunov-Schmidt.

%and consider solutions in  $D^{1,2}\left(\R^N\right)$ where $D^{1,2}\left(\R^N\right)=\{u\in L^{2^*}(\R^N)\hbox{ such that } |\na u| \in L^2(\R^N)\}$.\\

In this paper we follow a different approach that will allow us to obtain, among other results,  richer multiplicity results of solutions. The main ingredient of our proofs is given by the following 
map,
$$
{\mathcal L}_p:C(0,+\infty)\rightarrow C(0,+\infty)
$$
defined as
\begin{equation}\label{i1}
v(r)={\mathcal L}_p(u(r))=r^a u\left(r^b\right)\,\hbox{ for }r>0,\ p>1
\end{equation}
with
\begin{equation}\label{k}
a=2\frac{(N-2)(1-\nu_\l)}
{(p-1)(N-2)(\nu_\l-1)+4},
\end{equation}
and
\begin{equation}\label{alpha}
b=\frac4{(p-1)(N-2)(\nu_\l-1)+4}.
\end{equation}
where
\begin{equation}\label{ni-l}
\nu_\l=\sqrt{1-\frac {4\l}{(N-2)^2}}.
\end{equation}
Let $(0,T)\subset\R$ be an interval ($T=+\infty$ is allowed) and ${\mathcal D}^{1,2}((0,T),r^{N-1}dr)=\left\{u:(0,T)\rightarrow\R\hbox{ such that }\int_0^T
|u'|^2r^{N-1}dr
<+\infty\right\}$. The next proposition highlights the main properties of the operator ${\mathcal L}_p$.
\begin{proposition}\label{i2}
Let $p>1$, $\l<\frac{(N-2)^2}4$ and $u$ be a function satisfying
\begin{equation}\label{i3}
-u''-\frac{N-1}ru'-\frac {\l}{r^2}u=u^p\quad\hbox{in }(0,T)
\end{equation}
with $T\in(0,+\infty]$.
Then,  we have that  $v(r)={\mathcal L}_p(u(r))$ satisfies
\begin{equation}\label{i4}
-v''-\frac{M-1}rv'=A(\l,p)v^p\quad\hbox{in }(0,T^{\frac 1b})
\end{equation}
where
\begin{equation}\label{i6}
M-1=\frac{(p+3)(N-2)(\nu_\l-1)+4(N-1)}{(p-1)(N-2)(\nu_\l-1)+4}
\end{equation}
and
\begin{equation}\label{i5}
A(\l,p)=b^2=\left(\frac4{(p-1)(N-2)(\nu_\l-1)+4}\right)^2.
\end{equation}
If we choose
\begin{equation}\label{i1a}
T=+\infty\hbox{ when }p=\frac {N+2}{N-2}
\end{equation}
or
\begin{equation}\label{i1ba}
T=1\hbox{ when }1<p<\frac {N+2}{N-2},
\end{equation}
then we have that
\begin{equation}\label{i2a}
{\mathcal L}_p:{\mathcal D}^{1,2}((0,T),r^{N-1}dr)\rightarrow{\mathcal D}^{1,2}((0,T),r^{M-1}dr)
\end{equation}
and
\begin{equation}\label{i2b}
||{\mathcal L}_pu||_{{\mathcal D}^{1,2}((0,T),r^{M-1}dr)}=
\frac 1{\nu_\l}\int_0^T
\left(u'(s)^2-\frac\l{s^2}u^2(s)\right)s^{N-1}ds
\end{equation}
\end{proposition}
The previous proposition establishes a one-to-one relationship between the radial solutions to \eqref{first} and the ODE \eqref{i4}. This allows us to find some old and new results about radial solutions to \eqref{first}.\\
On the other hand we stress that the map ${\mathcal L}_p$ will be used also to establish existence of $nonradial$ solutions.\\
In this paper we analyze two different situations: either 
\begin{equation}
p=\frac{N+2}{N-2} \hbox{ and }\Om=\R^N,
\end{equation}
or
\begin{equation}
1<p<\frac{N+2}{N-2} \hbox{ and }\Om=B_1.
\end{equation}
Let us start considering $p=\frac{N+2}{N-2}$ so that \eqref{first} becomes
\begin{equation}\label{1}
\left\{\begin{array}{ll}
-\Delta u-\frac {\l}{|x|^2}u=C(\l)u^\frac{N+2}{N-2} & \hbox{ in }\R^N\\
u\geq 0\\
u\in D^{1,2}\left(\R^N\right)
\end{array}\right.
\end{equation}
where $N\geq 3$, $D^{1,2}\left(\R^N\right)=\{u\in L^{2^*}(\R^N)\hbox{ such that } |\na u| \in L^2(\R^N)\}$ and $C(\l)=N(N-2)\left(1-\frac {4\l}{(N-2)^2}\right)$
(we have added the constant $C(\l)$ just to have a simpler expression of the explicit radial solutions).\\
Our starting point is the paper \cite{T} of Terracini.  For what concerns the radial case, Terracini shows that the unique radial solutions of \eqref{1} in $D^{1,2}\left(\R^N\right)$ are given by the functions
\begin{equation}\label{sol-rad}
u_{\d,\l}(r)=\frac {r^{\frac{N-2}2\left(\nu_\l-1\right)}\d^{\frac{N-2}2}}{\left(1+\d^{2}r^{2\nu_\l}\right)^\frac{N-2}2}
\end{equation}
with $\nu_\l$ %=\sqrt{1-\frac {4\l}{(N-2)^2}}$ 
as in \eqref{ni-l}.
Moreover she proved the 
following  result.
\begin{theorem}[Terracini]
Let $\l\in[0,\frac{2N}{N-2})$. Then problem \eqref{1} has a unique (up to rescaling) solution in $D^{1,2}\left(\R^N\right)$.
Moreover there exists $\l^*<0$ such that for $\l<\l^*$ problem \eqref{1} admits at least two positive solutions in $D^{1,2}\left(\R^N\right)$ which are distinct modulo rescaling. One is radial while the second is not.
\end{theorem}
Another existence result was obtained some years later by Jin, Li and Xu (\cite{JLX}) where the authors proved the existence of singular solutions of the form $u(r,\theta)=r^\frac{2-N}2g(\theta)$, with $(r,\theta)$ polar coordinates in $\R^N$.\par
Finally, we recall a result by Musso and Wei (\cite{MW}) where it was proved the existence of infinitely many solutions for any $\lambda<0$. Note that {\em the energy }of this solutions, namely the quantity $E(u)=\frac12\int_{\R^N}\left(|\nabla u|^2-\frac {\lambda}{|x|^2}u^2\right)-\frac{N-2}{2N}\int_{\R^N}|u|^\frac{2N}{N-2}$ is large.

\noindent The results in \cite{T} are based on the moving plane method (when $\l>0$) and on the analysis of the radially symmetric case in the phase space.\\
Using the map ${\mathcal L}_{\frac{N+2}{N-2}}$ 
%$$\int_D|\nabla v|^2=\frac2{2+\a}\left(\int_D|\nabla u|^2-\frac\l{|x|^2}u^2\right).$$}
we give another proof of some results in \cite{T} in the radial case. In our opinion this approach is simpler. Actually, as showed in Proposition \eqref{i2}, since in this case we have that $M=N$ then
the map ${\mathcal L}_{\frac{N+2}{N-2}}$ reduces the study of the radial solutions of \eqref{1} to the well known problem,
\begin{equation}\label{eq-critico}
\left\{\begin{array}{ll}
-\Delta U=N(N-2)U^{2^*-1} & \hbox{ in }\R^N\\
U\geq 0\\
U\in D^{1,2}\left(\R^N\right).
\end{array}\right.
\end{equation}
Solutions of \eqref{eq-critico}
have been completely classified in  \cite{CGS}, where the authors proved that the solutions are given by
\begin{equation}\label{sol-critico-reg}
U_{\d}(r)=\frac {\d^{\frac{N-2}2}}{\left(1+\d^{2}r^2\right)^\frac{N-2}2}
\end{equation}
with $\delta>0$
and they are extremal functions for the well-known Sobolev inequality,
\begin{equation}\label{i3a}
\int_{\R^N}|\nabla u|^2\ge S\left(\int_{\R^N}|u|^\frac{2N}{N-2}\right)^\frac{N-2}N
\end{equation}
for $u \in D^{1,2}\left(\R^N\right)$ and $S$ the best Sobolev constant.\\
In this way we derive that the function $u_{\d,\l}$ in \eqref{sol-rad} are the unique radial solutions to \eqref{1}  (see Corollary \ref{C1}) and some inequalities involving the Hardy norm (see Proposition \ref{teo-2}) (these results were proved in Section 4 in \cite{T} using the phase plane).

As we pointed out, the role of map ${\mathcal L}_{\frac{N+2}{N-2}}$ is not restricted only to the radial case. Indeed it can be used to characterize {\em all}  solutions of the linearized problem at $u_{\d,\l}$, namely
\begin{equation}\label{lin}
\left\{\begin{array}{ll}
-\Delta v-\frac {\l}{|x|^2}v=N(N+2)\nu_\l ^2u_{\d,\l}^\frac4{N-2}v & \hbox{ in }\R^N\\
v\in D^{1,2}\left(\R^N\right)
\end{array}\right.
\end{equation}
Next result classifies the solution to \eqref{lin},
\begin{lemma}\label{lemma-lin}
Let $\l< \frac{(N-2)^2}4$ and 
\begin{equation}\label{lambda-j}
\l_j=\frac{(N-2)^2}4\left(1-\frac {j(N-2+j)}{N-1}\right),\ j\in\mathbb {N}.
\end{equation}
If $\l\neq \l_j$ then the space of solutions of \eqref{lin} (with $\d=1$) has dimension 1 and it is spanned by
\begin{equation}\label{Z}
Z_\l(x)= \frac{ |x|^{\frac{N-2}2\left(\nu_\l-1\right)}\left(1-|x|^{2\nu_\l}  \right)}{\left( 1+ |x|^{2\nu_\l}\right)^{\frac N2}}
\end{equation}
where $\nu_\l$ is as defined in \eqref{ni-l}.\\
If else $\l=\l_j$ for some $j=1,\dots$ then the space of solutions of \eqref{lin} (with $\d=1$) has dimension $1+\frac{(N+2j-2)(N+j-3)!}{(N-2)!\,j!} $ and it is spanned by
\begin{equation} \label{i13}
Z_{\l_j}(x) \,\, , \,\, Z_{j,i}(x)=\frac{|x|^{\frac N2\nu_{\l_j}- \frac{N-2}2} Y_{j,i}(x)}{ \left( 1+ |x|^{2\nu_{\l_j}}\right)^{\frac N2}  }
\end{equation}
where $\{Y_{j,i}\}$, ${i=1,\dots,\frac{(N+2j-2)(N+j-3)!}{(N-2)!j!}}$, form a basis of $\mathbb{Y}_j(\R^N)$, the space of all homogeneous harmonic polynomials of degree $j$ in $\R^N$.
\end{lemma}
%\begin{remark}
%Note that $\l_1=0$ and $\l_j<0$ for $j\ge2$. Since the case $\l_1=0$ corresponds to the well known problem \eqref{eq-critico} we only consider the case $j\ge2$.
%\end{remark}
A consequence of the previous result is the computation of the Morse index of $u_{\l}=u_{1,\l}$.
\begin{proposition}\label{cor-1}
Let $u_{\l}:=u_{1,\l}$ be the radial solution of \eqref{1}, then its Morse index $m(\l)$ is equal to
\begin{equation}\label{morse-index}
m(\l) = \sum_{0 \leq  j<\frac{2-N}2+\frac1 2 \sqrt{N^2-\frac {16(N-1) \l}{(N-2)^2}} \atop_{j\  integer} } \frac{(N+2j-2)(N+j-3)!}{(N-2)!\,j!}.
\end{equation}
In particular, we have that the Morse index of $u_{\l}$ changes as $\l$ crosses the values $\l_j$ and also that $m(\l) \to+\infty$ as $\l \to-\infty.$
\end{proposition}
Next aim is to obtain multiplicity results of nonradial solutions as $\l<0$ bifurcating from the radial solution $u_{\l}:=u_{1,\l}$ in \eqref{sol-rad}. 
%One important step in applying the bifurcation theory is the choice of a suitable functional space that allows to compute the Leray-Schauder degree of the associated operator.\\

For any $h\in\N$ we let $O(h)$ be the orthogonal group in 
$\R^h$. Our main result for problem \eqref{1} is the following (see \eqref{X} for the definition of the space $X$). 
\begin{theorem}\label{bif}
Let us fix $j\in \N$ and let $\l_j$ as in \eqref{lambda-j}. Then\\
i) If $j$ is odd there exists at least a continuum of nonradial weak solutions to \eqref{1}, invariant with respect to $O(N-1)$, bifurcating from the pair $(\l_j,u_{\l_j})$ in $(-\infty,0)\times X$.\\
ii) If $j$ is even there exist at least $\big[\frac N2\big]$ continua of nonradial weak solutions to
\eqref{1} bifurcating from the pair $(\l_j,u_{\l_j})$ in $(-\infty,0)\times X$. The first branch is $O(N-1)$ invariant, the second is $O(N-2)\times O(2)$ invariant, etc.\\
Moreover all these solutions $v_{\l}$ satisfy
$$\sup_{x\in \R^N}(1+|x|)^{\gamma}|v_{\l}(x)|\leq C$$
where $\gamma\in \R$ satisfies $\frac{N-2}2<\gamma<N-2$.
\end{theorem}
\begin{remark}
The solutions of Theorem \ref{bif} are different from the nonradial one founded in \cite{T}. Indeed, the nonradial solution $\bar u$ in \cite{T} satisfies 
\begin{equation}\label{i8}
\frac{\int_{\R^N}\left(|\nabla\bar u|^2-\frac\l{|x|^2}\bar u^2\right)}{\left(\int_{\R^N}|\bar u|^\frac{2N}{N-2}\right)^\frac{N-2}N}<k^\frac2NS<\left(1-\frac{4\l}{(N-2)^2}\right)^\frac{N-1}N S
\end{equation}
for some integer $k$ and $\lambda$ large enough.\\
On the other hand, since 
\begin{equation}\label{i9}
\frac{\int_{\R^N}\left(|\nabla u_{\l_j}|^2-\frac\l{|x|^2}u_{\l_j}^2\right)}{\left(\int_{\R^N}|u_{\l_j}|^\frac{2N}{N-2}\right)^\frac{N-2}N}=\left(1-\frac{4\l_j}{(N-2)^2}\right)^\frac{N-1}NS,
\end{equation}
our continuum of solutions does not contain $\bar u$, at least in a region "close" to the radial branch $u_{\l}$.\par
The same remark applies to the solutions founded in \cite{MW} since there energy is bigger that $nS$, for some large integer $n$.
\end{remark}
\begin{remark}
A consequence of the previous remark is that, for $\l$ close to $\l_j$ and $j$ large, problem \eqref{1} has at least {\em three} solutions. One of them is the radial function $u_{\l}$ in \eqref{sol-rad} and the others are nonradial functions. Moreover, if $j$ is an even number sufficiently large, we have at least $\big[\frac N2\big]+2$ solutions. This shows that the bifurcation diagram of the solutions to \eqref{1} is very complicated for $\l<0$ and, of course, quite difficult to describe.
\end{remark}
\medskip
\noindent We now consider the case $1<p<\frac{N+2}{N-2}$ and the problem 
\begin{equation}\label{i10}
\left\{\begin{array}{ll}
-\Delta u-\frac {\l}{|x|^2}u=u^p & \hbox{ in }B_1\\
u>0& \hbox{ in }B_1\\
%u=0& \hbox{ on }\partial B_1
u\in H^{1}_0(B_1)
\end{array}\right.
\end{equation}
where $B_1$ is the unit ball in $\R^N$.  A complete description of \eqref{i10} for $\l\geq 0$ was given in \cite{CG} where the authors proved that,  problem  \eqref{i10} admits a unique solution which is radial. We are not aware of any result instead for $\l<0$. Anyway, since the problem is subcritical then the existence of a radial solution is a straightforward consequence of the Mountain Pass Theorem.
%, for $\l\ge0$, all solutions to \eqref{i10} are radial. Moreover the radial solution is unique. 
%For this reason we only consider the case $\l<0$.\\
Likewise to the previous case, Proposition \eqref{i2} shows that  the map ${\mathcal L}_p$ provides an equivalence between the radial solutions to \eqref{i10} and the solutions of 
\begin{equation}\label{i11}
\left\{\begin{array}{ll}
-v''-\frac{M-1}rv'=A(\l,p)v^p & \hbox{ in }(0,1)\\
v>0& \hbox{ in }(0,1)\\
v'(0)=v(1)=0
\end{array}\right.
\end{equation}
with $M$ and $A(\l,p)$ as in \eqref{i6}, \eqref{i5} (see Section \ref{s6} for a discussion about the boundary conditions). It is known that this problem has a unique solution that we call $v_\l$.
Then we get the following result:
\begin{theorem}\label{d2}
Let $\l<\frac{(N-2)^2}4$ and $1<p<\frac{N+2}{N-2}$. Then the problem \eqref{i10}
admits only one radial solution $u_\l(r)$. Moreover 
\begin{equation}\label{origin}
r^{\frac{(N-2)}2 (1-\nu_\l)}u_\l(r)\to C \quad \hbox{ as } r\to 0^+
\end{equation}
where $C>0$.
\end{theorem}  
\noindent {This result extends the uniqueness result of \cite{CG} to the case $\l<0$ and shows that radial solutions to \eqref{i10} satisfy $u_\l(0)=0$ for $\l<0$ differently from the case of $\l\geq 0$ where $u_\l\notin L^{\infty}(B_1)$. Then the monotonicity properties of the Gidas, Ni and Nirenberg theorem cannot hold when $\l<0$.}\\ 
Once we have this branch of radial solution $u_\l$ for any $\l<0$ we can look for nonradial solutions that arise by bifurcation.
The strategy to obtain multiplicity results for $\l<0$ is the same as in the case $p=\frac{N+2}{N-2}$. %first we prove the existence of a branch of radial solution for any $\l<0$ satisfying suitable properties like uniqueness and 
First we prove non degeneracy of $u_\l$ in the space of radial function. %In Corollary \ref{C2} it is proved that any function $u_\l$ belonging to this branch satisfies $u_\l(0)=0$. So the monotonicity properties of the Gidas, Ni and Nirenberg theorem cannot hold.\\
Then we characterize the values of $\l$ for which the linearized operator at the radial solution $u_\l$ is non invertible and we compute the change of the Morse index of the radial solutions at these points. These values of $\l$ are related to a weighted eigenvalue for problem \eqref{i11}. %Finally we apply a Leray-Schauder degree argument tha gives the bifurcation.\\
To this end we let $v_\l\in  H^1((0,1),r^{M-1}dr)$  be the unique solution to \eqref{i11} and set
\begin{equation}\label{i12}
\L(\l)=\inf\limits_{w\in H^1((0,1),r^{M-1}dr),\ w(1)=0}\frac{\int_0^1\left(|w'|^2-pA(\l,p)v_\l^{p-1}w^2
\right)r^{M-1}}{\int_0^1w^2r^{M-3}}.
\end{equation}
The infimum $\L(\l)$ is well defined by the Hardy inequality but it is not clear if it is achieved. Indeed the embedding of $  H^1((0,1),r^{M-1}dr)   \hookrightarrow L^2((0,1),r^{M-3}dr)$ is not compact.
However the crucial information that the infimum \eqref{i12} is strictly negative implies that it is attained.  This is proved in Proposition \ref{d12-bis} in the Appendix in a more general case and relies on a careful study of some weighted problem given in \cite{GGN2}, Section 2.\par
Now we can state the following result,
\begin{theorem}\label{i14}
For any $j\in \N$, $j\geq 1$ there exists a value $\l_j$ that satisfies
\begin{equation}
\left(\frac{4}{(p-1)\left(2-N+\sqrt{(N-2)^2-4\l}\right)+4}\right)^2j(N-2+j)=-\L(\l)
\end{equation}
and an interval $I_j\subset (-\infty, 0)$ such that $\l_j\in I_j$
and that a nonradial bifurcation occurs at $(\l,u_{\l})$ for $\l\in I_j$.\\
Moreover, if $j$ is even there exist at least $\big[\frac N2\big]$ continua of nonradial solutions to
\eqref{i10} bifurcating from $(\l,u_{\l})$ for $\l\in I_j$. The first branch is $O(N-1)$ invariant, the second is $O(N-2)\times O(2)$ invariant, etc.
\end{theorem}

\bigskip 
%Let us fix $j\in \N$ and let $\l_j$ be a value that satisfies 
%\begin{equation}
%\left(\frac{4}{(p-1)\left(2-N+\sqrt{(N-2)^2-4\l}\right)+4}\right)^2j(N-2+j)=-\L_{1}(\l)
%\end{equation}
%Let $u_{\l}:={\mathcal L}_p^{-1}\left(v_{\l}(r)\right)$ be the unique radial solution to \eqref{i10}. Then\\
%i) If $j$ is odd there exists at least a continuum of nonradial solutions to \eqref{i10} for  $1<p<\frac{N+2}{N-2}$, invariant with respect to $O(N-1)$, bifurcating from the pair $(\l_j,u_{\l_j})$.\\
%ii) If $j$ is even there exist at least $\big[\frac N2\big]$ continua of nonradial solutions to
%\eqref{i10} bifurcating from the pair $(\l_j,u_{\l_j})$. The first branch is $O(N-1)$ invariant, the second is $O(N-2)\times O(2)$ invariant, etc.\\
%Moreover the bifurcation is global and the Rabinowitz alternative Theorem holds.

The paper is organized as follows: in Section \ref{s1} we show the main properties of the map ${\mathcal L}_p$. In Section \ref{S2} we consider the case $p=\frac{N+2}{N-2}$ and in Section \ref{s6} the sub critical case $1<p<\frac{N+2}{N-2}$. In the Appendix we prove some technical result.
\sezione{Main properties of the map ${\mathcal L}_p$}\label{s1}
In this section we give the proof of Proposition \ref{i2}.
\begin{proof}[Proof of Proposition \ref{i2}.]
Let $u(r)$ be a solution of \eqref{i3}. Then a straightforward computation shows that $v(r)={\mathcal L}_p(u(r))$ is a solution to \eqref{i4} with $M$ and $A(\l,p)$ satisfying \eqref{i6} and \eqref{i5} respectively.\\
Now let us show \eqref{i2b}. We just consider the case $p=\frac{N+2}{N-2}$ and $T=+\infty$ (the subcritical case $1<p<\frac{N+2}{N-2}$  is similar and easier).\\
Note that in this case we have $a=\frac {(N-2)(1-\nu_\l)}{2\nu_\l}$, $b=\frac 1{\nu_\l}$
and 
\begin{equation}\label{second}
v(r)=r^{\frac{(N-2)(1-\nu_\l) }{2\nu_\l }} u\left(r^{\frac 1{\nu_\l}}\right)\,\hbox{ for }r>0.
\end{equation}
First of all we observe that since $u\in{\mathcal D}^{1,2}((0,+\infty),r^{N-1}dr)$ then 
$\int_0^{+\infty}u(s)^\frac{2N}{N-2}s^{N-1}<+\infty$ and so there exist sequences $\d_n\rightarrow0$ and $M_n\rightarrow+\infty$ such that
\begin{equation}\nonumber
u(\d_n)\d_n^\frac{N-2}2\rightarrow0\quad\hbox{and}\quad u(M_n)M_n^\frac{N-2}2\rightarrow0.
\end{equation}
By \eqref{second} we derive that
\begin{equation}\nonumber
\frac{r^{\frac{N}{\nu_\l}-N}}{\nu_\l^2}\left(
u'\left(r^{\frac 1{\nu_\l}}\right)\right)^2=\frac{(N-2)^2(1-\nu_\l)^2}{4\nu_\l^2}
\frac{v^2(r)}{r^2}-\frac{(N-2)(1-\nu_\l)}{\nu_\l}\frac{v(r)v'(r)}r+\left(v'(r)\right)^2
\end{equation}
Choosing $\e_n=\d_n^{\nu_\l}$ and $R_n=M_n^{\nu_\l}$  and integrating we get
\begin{eqnarray}\label{2.5}
&&\frac1{\nu_\l^2}\int_{\e_n}^{R_n}\left(
u'\left(r^{\frac 1{\nu_\l}}\right)\right)^2r^{\frac{N}{\nu_\l}-1}=\frac{(N-2)^2(1-\nu_\l)^2}{4\nu_\l^2}\int_{\e_n}^{R_n}
v^2(r)r^{N-3}\nonumber\\
&&-\frac{(N-2)(1-\nu_\l)}{\nu_\l}\int_{\e_n}^{R_n}v(r)v'(r)r^{N-2}+\int_{\e_n}^{R_n}\left(v'(r)
\right)^2r^{N-1}.
\end{eqnarray}
Then we have, using again \eqref{second}
\begin{eqnarray}
&&\int_{\e_n}^{R_n}v(r)v'(r)r^{N-2}=\frac12r^{N-2}v^2(r)\Big|_{\e_n}^{R_n}-\frac{N-2}2\int_{\e_n}^{R_n}v^2(r)r^{N-3}=\nonumber\\
&&\frac12\left(u(M_n)M_n^\frac{N-2}2\right)^2-\frac12\left(u(\d_n)
\d_n^\frac{N-2}2\right)^2-\frac{N-2}2\int_{\e_n}^{R_n}v^2(r)r^{N-3}.\nonumber
\end{eqnarray}
Hence, by the choice of $\e_n$ and $R_n$ we deduce that
\begin{equation}\nonumber
\int_0^{+\infty}v(r)v'(r)r^{N-2}=-\frac{N-2}2\int_0^{+\infty}v^2(r)r^{N-3}.
\end{equation}
So \eqref{2.5} becomes
\begin{eqnarray}
&&\frac 1{\nu_\l}\int_0^{+\infty}\left(
u'(s)\right)^2s^{N-1}=\frac{(N-2)^2(1-\nu_\l)}{4\nu_\l^2}(1+\nu_\l)\int_0^{+\infty}
v^2(r)r^{N-3}+\nonumber\\
&&\int_0^{+\infty}\left(v'(r)
\right)^2r^{N-1}\hbox{(recalling the definition of $\nu_\l$)}=\nonumber\\
%&&\frac{\l}{\nu_\l^2}\int_0^{+\infty}
%v^2(r)r^{N-3}+\int_0^{+\infty}\left(v'(r)
%\right)^2r^{N-1}=\nonumber\\
&&\frac{\l}{\nu_\l}\int_0^{+\infty}
u^2(s)s^{N-3}+\int_0^{+\infty}\left(v'(r)
\right)^2r^{N-1},\nonumber
\end{eqnarray}
which gives the claim.
\end{proof}
\sezione{The critical case $p=\frac{N+2}{N-2}$}\label{S2}
\subsection{Basic properties and the main inequality}\label{s2}
In this section we consider problem \eqref{1}.
First let us observe that if we put $p=\frac{N+2}{N-2}$ in \eqref{i1}-\eqref{ni-l} we get
\begin{equation}
a=\frac{(N-2)(1-\nu_\l)}{2\nu_\l},
\end{equation}
and
\begin{equation}
b=\frac1{\nu_\l}.
\end{equation}
Moreover, if $u\in{\mathcal D}^{1,2}((0,+\infty),r^{N-1}dr)$ satisfies in weak sense
\begin{equation}\label{criticou}
-u''-\frac{N-1}ru'-\frac{\l}{r^2}u=C(\l)u^{2^*-1} \quad \hbox{ in }(0,+\infty)
\end{equation}
where $C(\l)=N(N-2)\nu_\l^2$ and $\nu_\l$ as in \eqref{ni-l}, then
Proposition \ref{i2} shows that $v\in{\mathcal D}^{1,2}((0,+\infty),r^{N-1}dr)$ weakly solves
\begin{equation}\label{criticov}
-v''-\frac{N-1}{r}v'=N(N-2)v^{2^*-1}\quad \hbox{ in }(0,+\infty).
\end{equation}
\begin{corollary}\label{C1}
All the radial solutions in $D^{1,2}(\R^N)$ of \eqref{1} are given by the functions $u_{\d,\l}(r)$ in \eqref{sol-rad}.
\end{corollary}
\begin{proof}
It follows directly by \eqref{criticou} and \eqref{criticov}. Since all solution to \eqref{criticov} are given by $v_\d(r)=\frac {\d^{\frac{N-2}2}}{\left(1+\d^{2}r^2\right)^\frac{N-2}2}$,
by the definition of ${\mathcal L}_p$ we deduce that
\begin{equation}
u_{\d,\l}(r)=r^{-a}v_\d\left(r^\frac1b\right)=r^{\frac{(N-2)(\nu_\l-1)}{2\nu_\l}}v_\d\left(r^{\nu_\l}\right)=\frac {r^{\frac{(N-2)(\nu_\l-1)}{2\nu_\l}}\d^{\frac{N-2}2}}{\left(1+\d^{2}r^{2\nu_\l}\right)^\frac{N-2}2}
\end{equation}
which gives the claim.
\end{proof}
Now we prove an interesting inequality. We remark that, in the case $\l\ge0$ this is basically contained in \cite{T}. If $\l<0$ we do not find any references although this can be shown using (for example) the concentration-compactness principle of Lions. Anyway, we think that the interest of the next proposition is in its proof, which reduces the Hardy inequality to the classical Sobolev imbedding.
\begin{proposition}\label{teo-2}
Let $\l<\frac{(N-2)^2}4$. Then we have that for any radial function $u\in D^{1,2}\left(\R^N\right)$
\begin{equation}\label{z2}
\int_{\R^N}\left(|\nabla u|^2-\frac\l{|x|^2}u^2\right)\ge\left(1-\frac{4\l}{(N-2)^2}\right)^\frac{N-1}NS\left(\int_{\R^N}|u|^\frac{2N}{N-2}\right)^\frac{N-2}N
\end{equation}
where $S$ is the best Sobolev constant. Moreover the previous inequality is achieved only for $u(r)=u_{\d,\l}(r)$.\\
If $\l>0$ then \eqref{z2} holds for any $u\in D^{1,2}\left(\R^N\right)$.
\end{proposition}
\begin{proof}[Proof of Proposition \ref{teo-2}]
Let $v$ be as in \eqref{second}. Then by \eqref{i2b} we get
\begin{eqnarray}
&&\int_{\R^N}\left(|\nabla u|^2-\frac\l{|x|^2}u^2\right)=\nu_\l\int_{\R^N}|\nabla v(|x|)|^2\ge\nu_\l S\left(\int_{\R^N}|v|^\frac{2N}{N-2}\right)^\frac{N-2}N
\nonumber\\
&&=\nu_\l S\left(\omega_N\int_0^{+\infty}|v(r)|^\frac{2N}{N-2}r^{N-1}\right)^\frac{N-2}N=\nu_\l^{2\frac{N-1}N}S\left(\omega_N\int_0^{+\infty}|u(s)|^\frac{2N}{N-2}s^{N-1}\right)^\frac{N-2}N=
\nonumber\\
&&\left(1-\frac{4\l}{(N-2)^2}\right)^\frac{N-1}NS\left(\int_{\R^N}|u|^\frac{2N}{N-2}\right)^\frac{N-2}N \nonumber
\end{eqnarray}
which gives the claim. Note that the previous inequality becomes an identity if and only if $\int_{\R^N}|\nabla v(|x|)|^2=S\left(\int_{\R^N}|v|^\frac{2N}{N-2}\right)^\frac{N-2}N$. It is well known (see for example \cite{CGS}) that this implies $v(x)=\frac {\d^{\frac{N-2}2}}{\left(1+\d^{2}r^2\right)^\frac{N-2}2}$ for some positive $\d$. Recalling that  (see Corollary \ref{C1})
$u(r)=r^{-\frac{(N-2)(1-\nu_\l)}2}v\left(r^{\nu_\l}\right)$ we have the uniqueness of the minimizer.\\
Let us show that if $\l>0$ then \eqref{z2} holds for any $u\in D^{1,2}\left(\R^N\right)$. This follows by the classical spherical rearrangement theory. Indeed, denoting by $u^*=u^*(|x|)$ the classical Schwartz rearrangement we have that $\int_{\R^N}|\nabla u^*|^2\le\int_{\R^N}|\nabla u|^2$, $\int_{\R^N}|u^*|^\frac{2N}{N-2}=\int_{\R^N}|u|^\frac{2N}{N-2}$ and $\int_{\R^N}\frac{|u|^2}{|x|^2}\le\int_{\R^N}\frac{|u^*|^2}{|x|^2}$. Hence, if $\l>0$, we get
\begin{equation}\nonumber
\frac{\int_{\R^N}\left(|\nabla u^*|^2-\frac\l{|x|^2}|u^*|^2\right)}{\left(\int_{\R^N}|u^*|^\frac{2N}{N-2}\right)^\frac{N-2}N}\le
\frac{\int_{\R^N}\left(|\nabla u|^2-\frac\l{|x|^2}u^2\right)}{\left(\int_{\R^N}|u|^\frac{2N}{N-2}\right)^\frac{N-2}N}
\end{equation}
which implies the claim.
\end{proof}
\subsection{The linearized operator}\label{s3}
In this section we linearize problem \eqref{1} at the radial solution $u_{\l}:=u_{1,\l}$ 
given in \eqref{sol-rad} and we look for the degeneracy points. This is equivalent to find nontrivial solutions for the linearized problem \eqref{lin}. Using \eqref{sol-rad} we rewrite \eqref{lin} as follows
\begin{equation}\label{lin-2}
\left\{\begin{array}{ll}
-\Delta v-\frac {\l}{|x|^2}v=N(N+2)\nu_\l ^2 \frac{|x|^{2 \left(\nu_\l-1\right)}} {\left(1+ |x|^{2\nu_\l}\right)^ 2}     v & \hbox{ in }\R^N\\
v\in D^{1,2}\left(\R^N\right).
\end{array}\right.
\end{equation}
We solve \eqref{lin-2} using the decomposition along the spherical harmonic functions and we write
\begin{equation*}
v(r,\theta) = \sum_{j=0}^{\infty} \psi_j(r)Y_j(\theta), \qquad \text{where} \qquad r=|x| \,\, , \,\, \theta=\frac{x}{|x|} \in S^{N-1}
\end{equation*}
and
\begin{equation*}
\psi_j(r) = \int_{S^{N-1}}V(r,\theta)Y_j(\theta)d\theta.
\end{equation*}
Here $Y_j(\theta)$ denotes the $j$-th spherical harmonics, i.e. it satisfies
\begin{equation*}
-\Delta_{S^{N-1}} Y_j = \mu_jY_j
\end{equation*}
where $\Delta_{S^{N-1}}$ is the Laplace-Beltrami operator on $S^{N-1}$ with the standard metric and $\mu_j$ is the $j$-th eigenvalue
of $-\Delta_{S^{N-1}}$. It is known that
\begin{equation}\label{mu-j}
\mu_j = j(N-2+j), \qquad j=0,1,2,\dots
\end{equation}
whose multiplicity is
\begin{equation}\label{dim-mu-j}
 \frac{(N+2j-2)(N+j-3)!}{(N-2)!\,j!} 
\end{equation}
and that
$$ Ker\left(\Delta_{S^{N-1}} + \mu_j\right) = \mathbb{Y}_j(\R^N)\left|_{S^{N-1}}\right.$$
and $\mathbb{Y}_j(\R^N)$ is the space of all homogeneous harmonic polynomials of degree $j$ in $\R^N$.
The function $v$ is a weak solution of \eqref{lin-2} if and only if $\psi_j(r)$ is a weak solution of
\begin{align} \label{1.8}
\begin{cases}
-\psi_j''(r) - \frac{N-1}{r}\psi_j'(r) + \frac{\mu_j-\l}{r^2}\psi_j(r) = N(N+2) \nu_{\l}^2%\left(1-\frac{4\l}{(N-2)^2}\right)
\frac{r^{2 \left(\nu_\l%\sqrt{1-\frac {4\l}{(N-2)^2}}
-1\right)}} {\left(1+ r^{2\nu_\l%\sqrt{1-\frac {4\l}{(N-2)^2}}
}
\right)^ 2} \psi_j&
\text{in} \,\,(0,\infty) \\
\psi'_j(0)=0 \hbox{ if }j=0 \quad \hbox{ and }\psi_j(0)=0 \hbox{ if }j\geq 1\\
\psi_j \in D^{1,2}((0,+\infty),r^{N-1}dr).
\end{cases}
\end{align}
Letting, as in  \eqref{i1}, 
$$\hat \psi_j(r)=r^{\frac {(N-2)(1-\nu_\l)}{2\nu_\l}} \psi_j\left(  r^{\frac 1{\nu_\l}}\right),$$
we have that $\hat \psi_j$ weakly solves
\begin{align} \label{3.5}
\begin{cases}
-\hat \psi_j''(r) - \frac{N-1}{r}\hat \psi_j'(r) + \frac{\mu_j}{\nu_\l^2r^2}\hat \psi_j(r) = N(N+2)
\frac{1} {\left(1+ r^2\right)^ 2} \hat \psi_j&
\text{in} \,\,(0,\infty) \\
 \hat \psi_j \in D^{1,2}((0,+\infty),r^{N-1}dr).
\end{cases}
\end{align}
Problem \eqref{3.5} is well known: it comes from the linearization of the solution $U_1$ in \eqref{sol-critico-reg} of problem \eqref{eq-critico}.
This equation has a nontrivial solution (since $\frac{\mu_j}{\nu_\l^2}\geq 0$) if and only if one of the following holds
\begin{itemize}
\item[i)] $\frac{\mu_j}{\nu_\l^2}=0$
\item [ii)] $\frac{\mu_j}{\nu_\l^2}=N-1$.
\end{itemize}
Case $i)$ corresponds to the radial degeneracy, i.e. $j=0$. Equation \eqref{3.5} has the solution $\hat \psi_0(r)= \frac{1-r^2}{(1+r^2)^{\frac N2}}$ and turning back to \eqref{1.8} we get 
$$\psi_0(r)= \frac{ r^{\frac{N-2}2\left(\nu_\l-1\right)}\left(1-r^{2\nu_\l}  \right)}{\left( 1+ r^{2\nu_\l}\right)^{\frac N2}}$$
which is a solution to \eqref{1.8} for any $\l<\frac{(N-2)^2}4$. This proves \eqref{Z}.\\
When $\frac{\mu_j}{\nu_\l^2}=N-1$ then equation \eqref{3.5} has the solution $\hat \psi_j(r)=\frac r{(1+r^2)^\frac N2}$. Turning back to \eqref{1.8} we get 
that \eqref{1.8} has the solution 
$$\psi_j(r)=\frac{ r^{\frac{N-2}2\left(\nu_\l-1\right)+\nu_{\l}}}{\left( 1+ r^{2\nu_\l}\right)^{\frac N2}}$$
when ii) is satisfied.
Then $ii)$ implies that \eqref{1.8} has the solution $\psi_j(r)$ if and only if $\l=\l_j$
\begin{equation}\nonumber
\l_j=\frac{(N-2)^2}4\left(1-\frac {\mu_j}{N-1}\right)
\end{equation}
as in \eqref{lambda-j}. This proves \eqref{i13} and finishes the proof of Lemma \ref{lemma-lin}.\\[.4cm]
A first consequence of Lemma \ref{lemma-lin} is the computation of the Morse index of the solution $u_{\l}$ given in Proposition \eqref{cor-1}.
\begin{proof}[Proof of Proposition \eqref{cor-1}]
As shown in the Appendix (Corollary \ref{d12-mi}), the Morse index of the radial solution $u_\l$  is given by the number of negative values $\L_i$ such that the problem
%of the problem, counted with their multiplicity,
\begin{equation}\label{aut}
\left\{\begin{array}{ll}
-\Delta w-\frac {\l}{|x|^2}w-N(N+2)\nu_\l^2 \frac{|x|^{2 \nu_\l}} {\left(1+ |x|^{2\nu_\l}\right)^ 2}     w=\frac {\L}{|x|^2}w  & \hbox{ in }\R^N\\
w\in D^{1,2}\left(\R^N\right)
\end{array}\right.
\end{equation}
admits a weak solution $w_i$, counted with their multiplicity.
We denote by $w_i$ the solution of \eqref{aut} related to a negative value $\L_i$. We argue as before setting $w_{i,j}(r)=\int_{S^1}w_i(r,\theta)Y_j(\theta)\,d\theta$ and  $\widehat w_{i,j}(r)=r^{\frac{(N-2)(1-\nu_\l)}{2\nu_\l}}w_{i,j}(r^{\frac 1{\nu_\l}})$. Then $\widehat w_{i,j}(r)$ weakly satisfies
\begin{align} \label{aut-trans}
\begin{cases}
-\widehat w_{i,j}''(r) - \frac{N-1}{r}\hat w_{i, j}'(r)- N(N+2) 
\frac{1} {\left(1+ r^2\right)^ 2} \hat w_{i,j}= \frac{\L_i-\mu_j}{\nu_\l^2r^2}\hat w_{i,j}(r)
&\text{in} \,\,(0,\infty) \\
 \hat w_{i,j} \in D^{1,2}((0,+\infty),r^{N-1}dr).
\end{cases}
\end{align}
Since the problem 
\begin{align} \nonumber
\begin{cases}
-\eta''(r) - \frac{N-1}{r}\eta'(r)- N(N+2) 
\frac{1} {\left(1+ r^2\right)^ 2} \eta= \frac{\nu}{r^2}\eta
&\text{in} \,\,(0,\infty) \\
\eta \in D^{1,2}((0,+\infty),r^{N-1}dr)
\end{cases}
\end{align}
admits only one negative eigenvalue which is $1-N$, then
we derive that equation \eqref{aut-trans} has a nontrivial solution corresponding to a $\L_i<0$, if and only if 
\begin{equation}\nonumber
1-N=\frac{\L_i-\mu_j}{\nu_\l^2}.
\end{equation}
So we have that the indexes $j$ which contribute to the Morse index of the solution $u_\l$ are those that satisfy $\L_i=\nu_\l^2(1-N)+\mu_j<0$ and this implies, recalling the value of $\mu_j$  given in \eqref{mu-j}, % $(1-N)\left( 1-\frac {4\l} {(N-2)^2 }\right)+\mu_j<0$, i.e. 
$j<\frac{2-N}2+\frac 12 \sqrt{N^2-\frac {16(N-1)\l}{(N-2)^2}}$. Finally, %recalling that the value of $\mu_j$ is given in \eqref{mu-j} and 
using that the dimension of the eigenspace of the Laplace-Beltrami operator on $S^{N-1}$ related to $\mu_j$ is given in \eqref{dim-mu-j},  \eqref{morse-index} follows. 
\end{proof}
\begin{remark}\label{r1}
Reasoning as in the proof of the previous corollary it is easy to see that any eigenfunction of \eqref{aut} corresponding to an eigenvalue $\L<0$ can be written in the following way
$$w(r,\theta)= r^{\frac{(N-2)}{2}(\nu_\l-1)}\frac{r^{\nu_\l}}{\left(1+r^{2\nu_{\l}}\right)^{\frac N2}}Y_{j}(\theta)     $$
%\hat \psi(r^{\nu_\l})Y_{j}(\theta)     $$
where %$\hat \psi(r)=\frac r{(1+r^2)^{\frac N2}}$ and 
$Y_{j}(\theta)$ is a spherical harmonic related to the eigenvalue $\mu_j$.
\end{remark}
\subsection{The bifurcation result}\label{s4}
In this section we will start the proof Theorem \ref{bif} using the bifurcation theory.\\
To this end let us give some definitions. Let $\gamma>0$ be such that $\frac{N-2}2<\gamma<N-2$. For every $g\in L^{\infty}(\R^N)$ we define the weighted norm
\begin{equation}\label{norm-gamma}
\nor g\nor_{\gamma}:=\sup_{x\in \R^N}(1+|x|)^{\gamma}|g(x)|
\end{equation}
and the space $L^{\infty}_{\gamma}(\R^N):=\{g\in L^{\infty}(\R^N)\, \hbox{ such that } \exists C>0 \,\hbox{ and }\nor g\nor_{\gamma}<C\}$. Set 
\begin{equation}\label{X}
X=D^{1,2}(\R^N)\cap L^{\infty}_{\gamma}(\R^N).
\end{equation}
$X$ is a Banach space with the norm
\begin{equation}\label{norm-X}
\nor g\nor_X:=\max\{\nor g\nor_{1,2},\nor g\nor_{\gamma}\}
\end{equation}
where $\nor g\nor_{1,2}$ denotes the usual norm in $D^{1,2}(\R^N)$, i.e. $\nor g\nor_{1,2}=\left(\int_{\R^N}|\na g|^2\, dx\right)^{\frac 12}$.\\To apply the standard bifurcation theory we have to define a compact operator $T$ from $(-\infty,0)\times X$ into $X$ and to compute its Leray Schauder degree in $0$ in a suitable neighborhood of the radial solution $(\l,u_\l)$, at least for the values $\l\neq \l_j$. This seems difficult since the linearized operator (see \eqref{lin}) is not invertible due to the radial degeneracy of $u_\l$ for every $\l$ proved in Lemma \ref{lemma-lin}. 
%Hence we have to rule out the radial degeneracy of the  solutions  $(\l,u_\l)$ of \eqref{1} stated in Lemma \ref{lemma-lin}. 
To this end we define 
\begin{equation}\nonumber%\label{K_l}
K_\l:=\left\{v\in D^{1,2}(\R^N)\hbox{ such that }\int_{\R^N}u_{\l}^{2^*-2}vZ_\l\, dx=0\right\}
\end{equation}
with $Z_\l$ as defined in \eqref{Z}. Since $u_{\l}^{2^*-2}\in L^{\frac N2}(\R^N)$ and $v,\ Z_\l\in L^{2^*}(\R^N)$, we have that  $K_\l$ is a linear closed subspace of $D^{1,2}(\R^N)$. We let $P_\l$ be orthogonal the projection of $D^{1,2}(\R^N)
$ on $K_\l$.\\
Now we define the operator $T(\l,v): \left(-\infty,0\right)\times X\rightarrow K_{\l}\cap X$ as 
\begin{equation}\label{operator-T}
T(\l,v)=P_{\l}\left(\left(-\Delta-\frac{\l}{|x|^2}I\right)^{-1}\left(C(\l)(v^+)^{2^*-1}\right)\right)
\end{equation}
and look for
 zeros of the operator $I-T(\l,v)$. A function $v\in X$ is a zero of $ I-T(\l,v)$ if $v\in K_{\l}\cap X$ and $v$ is a weak solution of
\begin{equation}\label{molt}
%-Delta v-\frac{\l}{|x|^2}v-C(\l)v^{2^*-1}=LC(\l)\frac{N+2}{N-2}u_{\l}^{2^*-2}Z_\l 
-\Delta v-\frac{\l}{|x|^2}v-C(\l)v^{2^*-1}=L C(\l)\frac{N+2}{N-2}u_{\l}^{2^*-2}Z(x) \quad \hbox{ in }\R^N
\end{equation}
where $L=L(v)\in \R$ ($L$ is the Lagrange multiplier). The final step will be to show that $L=0$ so that $v$ is indeed a weak solution of \eqref{1}. This will be done in Section \ref{s5}.\\
Before proving our bifurcation result we need some technical results.
%First we need to prove the followig Lemma.
\begin{lemma}\label{l-T-ben-def}
The operator $ T(\l,v)$ is well defined from $\left(-\infty,0\right)\times X$ into $K_{\l}\cap X$.
\end{lemma}
\begin{proof}
It is enough prove that the operator 
\begin{equation}\label{T-tilde}
\widetilde T(\l,v)=\left(-\Delta -\frac{\l}{|x|^2}I\right)^{-1}\left(C(\l)(v^+)^{2^*-1}\right)
\end{equation} 
is  well defined from $\left(-\infty,0\right)\times X$ in $X$.\\
Since $(v^+)^{2^*-1}\in L^{\frac{2N}{N+2}}(\R^N)$ 
there exists a unique  $g\in D^{1,2}(\R^N)$ such that $g=\widetilde T(\l,v)$, see Lemma \ref{l1} in the Appendix, i.e. $g$ is a weak solution to
$$-\Delta g-\frac{\l}{|x|^2}g=C(\l)(v^+)^{2^*-1}\quad \hbox{ in }\R^N.$$
%Since $v\in X$ we have that 
%$$-\frac{C}{(1+|x|)^{\gamma \frac{N+2}{N-2}}}\leq -\Delta g-\frac{\l}{|x|^2}g\leq \frac{C}{(1+|x|)^{\gamma \frac{N+2}{N-2}}} \quad \hbox{ in }\R^N\setminus\{0\}.$$
Then, the comparison theorem for functions in $D^{1,2}(\R^N)$, yields
$$|g(x)|\leq C|w(x)|\quad \hbox{a. e. in }\R^N $$
where $w$ is the unique weak solution to
\begin{equation}\label{def-w}
-\Delta w-\frac{\l}{|x|^2}w=\frac{1}{(1+|x|)^{\gamma \frac{N+2}{N-2}}}\quad \hbox{ in }\R^N.
\end{equation}
We are going to prove that
\begin{equation}\label{n1}
\left(1+r\right)^\g w(r)\le C.
\end{equation}
To do this let $\bar w(r)=r^{k}w\left( r\right)$, where $k=\frac{N-2}2(1-\nu_\l)$ and $\nu_\l$ is as defined in \eqref{ni-l}. The function $\bar w$ weakly satisfies
$$-\bar w''-\frac{N-1-2k}r\bar w'=\frac {r^k}{(1+r)^{\gamma \frac{N+2}{N-2}}} \quad \hbox{ in }(0,+\infty).$$
Integrating we get
$$-r^{N-1-2k}\bar w'(r)= C+\int_{r_0}^r \frac{s^{N-1-k}}{(1+s)^{\gamma \frac{N+2}{N-2}}} \, ds$$
for any $r_0>0$.
Consequently $  -\bar w'(r)\leq Cr^{1-N+2k}+Cr^{1-N+2k+N-k-\gamma\frac{N+2}{N-2}} $ (we are assuming that $N-1-k-\gamma\frac{N+2}{N-2}\neq -1$; the case $N-k-\gamma\frac{N+2}{N-2}= 0$ follows in a very similar way).\\ 
 Since $ w\in D ^{1,2}(\R^N)$, from Ni's radial Lemma (see \cite{Ni}) we know that $ w(r)\leq Cr^{\frac{2-N}2}$ for any $r$, so that $\bar w(r)\leq Cr^{k-\frac {N-2}2}=Cr^{-\frac{ N-2}2\nu_\l}$. % where, as before $\nu_\l=\sqrt{1-\frac{4\l}{(N-2)^2}} >0$. 
Then $\bar w(r)\to 0$
 as $r\to +\infty$. Integrating $\bar w'(r)$ from $r$ to $+\infty$ yelds
$$\bar w(r)\leq Cr^{2-N+2k}+Cr^{2+k-\gamma\frac{N+2}{N-2}}.$$
This implies that, since by assumption $\frac{N-2}2<\gamma<{N-2}$ and $k<0$ for any $\l<0$,
\begin{equation}\label{w-infinito}
(1+r)^{\gamma}w(r)\leq Cr^{\gamma+2-N+k}+Cr^{\gamma+2-\gamma\frac{N+2}{N-2}}\leq C
\end{equation}
for $r$ large enough.\\% This shows \eqref{n1} . \\
To finish the proof of  \eqref{n1} we need to prove that $|w(x)|$ is bounded in a neighborhood of the origin. To this end we set
$$\widetilde w(r)=\frac 1{r^{N-2k-2}}\bar w\left(\frac 1r\right).$$
The function $\widetilde w$ is the Kelvin transform of $\bar w$ and so it satisfies
$$-\widetilde w''-\frac{N-1-2k}r \widetilde w'=\frac 1{r^{N-2k+2}}\frac {r^{-k}}{(1+\frac 1r)^{\gamma \frac{N+2}{N-2}}} \quad \hbox{ in }\R^N\setminus\{0\}.$$
Reasoning as before and integrating from $r_0$ to $r$ we get
$$-r^{N-1-2k}\widetilde w'(r)\leq C+C\int_{r_0}^rs^{-3-k}ds$$
and, assuming $-3-k>-1$,(observe that the case $-3-k<-1$ is easier and the case $-3-k=-1$ follows reasoning as in the first case)
$$-\widetilde w'(r)\leq Cr^{1-N+2k}+Cr^{1-N+2k-2-k}$$
for $r$ large enough. Then, using that $w(r)\rightarrow0$ as $r\rightarrow+\infty$,
$$\widetilde w(r)\leq  Cr^{2-N+2k}+Cr^{-N+k}$$
for $r$ large enough. This implies that
$$\bar w(r)=\frac 1{r^{N-2-2k}}  \widetilde w\left(\frac 1r\right )\leq Cr^{2+k}+C$$
for $ r$ small enough. Finally, using that $w(r)=r^{-k}\bar w(r)$ we have that
\begin{equation}\label{w-in0}
w(r)\leq C\left\{
\begin{array}{ll}
r^{-k}+r^{2} & \hbox{ if }k<-2\\
r^2(1-\log r)  & \hbox{ if }k=-2\\
r^{-k}  & \hbox{ if }-2<k<0
\end{array}
\right.
\hbox{  for }r \, \hbox{ small enough. }
\end{equation}
Estimates \eqref{w-infinito} and \eqref{w-in0} imply that
$$\sup_{x\in \R^N}(1+|x|)^{\gamma}|w(x)|\leq C$$
so that $w$ and hence $g$ belong to $L^{\infty}_{\gamma}(\R^N)$ concluding the proof. 
\end{proof}
\begin{proposition}\label{compattezza}
We have that:
\begin{itemize}
\item[i)] the operator $T(\l,v):\left(-\infty,0\right)\times X \rightarrow  K_{\l}\cap X $ defined in \eqref{operator-T}
is continuous with respect to $\l$  and it is compact from $ X $ into $K_{\l}\cap X$ for any $\l\in (-\infty,0)$ fixed;
\item[ii)] the linearized operator $I-T'_v(\l,u_{\l})I$ is invertible for any $\l\neq \l_j$, where $\l_j$ are as defined in \eqref{lambda-j}.
%\item[iii)] the Leray Schauder degree of the operator $I-T$ changes along the curve $(\l,u_\l)$ as $\l$ goes from $\l_n-\e$ to $\l_n+\e$ and $\e$ small enough.
\end{itemize}
\end{proposition}
\begin{proof}
Let us prove $i)$.
The operator $T(\l,v)$ is clearly continuous with respect to $\l$. As in the proof of Lemma \ref{l-T-ben-def}, we will prove that the operator $\widetilde T$, defined in \eqref{T-tilde} is compact 
from $X$ into $X$ for every $\l$ fixed. This implies in turn that $T$ is compact for every $\l$ fixed. To this end let $v_n$ be a sequence in $X$ such that $\nor v_n\nor_{X}\leq C$ and let $g_n= \widetilde T(\l, v_n)$. Then $g_n\in X$ and by Lemma \ref{l-T-ben-def} is a weak solution to
\begin{equation}\label{g_n}
-\D g_n-\frac{\l}{|x|^2}g_n=C(\l)(v_n^+)^{2^*-1}.
%-\Delta g_n- \frac{\l}{|x|^2}g_n=C(\l)v_n^{2^*-1}\quad \hbox{ in }\R^N\setminus \{0\}.
\end{equation}
%for any $\psi \in D^{1,2}(\R^N)$.
Since $v_n$ is bounded in $X$ then $|v_n(x)|\leq C(1+|x|)^{-\gamma}$ and $v_n$ is uniformly bounded in $D^{1,2}(\R^N)$ then, up to a subsequence, $v_n\rightharpoonup \bar v$ weakly in $D^{1,2}(\R^N)$ and almost everywhere in $\R^N$.
Multiplying  \eqref{g_n}  by $g_n$ and integrating we get
\begin{equation}\label{g-n-bis}
  \int_{\R^N}|\na g_n|^2\,dx -\int_{\R^N}\frac {\l}{|x|^2}g_n^2=C(\l)\int_{\R^N}(v_n^+)^{2^*-1}g_n\, dx.
\end{equation}
Then the Hardy and Sobolev inequalities imply that
$$c_{\l}\int_{\R^N}|\na g_n|^2\,dx\leq C\nor v_n^{2^*-1}\nor_{\frac{2N}{N+2}}\nor g_n\nor_{1,2}$$
where $c_{\l}$ is as in Lemma \ref{l2} %=1$ if $\l\leq 0$ and $c_\l=1-\frac{4\l}{(N-2)^2}$ if else $\l>0$,
and $\nor \cdot \nor_{q}$ denotes the usual norm in $L^{q }(\R^N)$. Then 
$$\nor g_n\nor_{1,2}\leq C$$
so that, up to a subsequence $g_n\rightharpoonup \bar g$ weakly in $D^{1,2}(\R^N)$ and almost everywhere in $\R^N$. 
Passing to the limit in \eqref{g_n}, we get that $\bar g$ is a weak solution of
$$-\Delta \bar g - \frac{\l}{|x|^2}\bar g=C(\l)(\bar v^+)^{2^*-1}\quad \hbox{ in }\R^N.$$
Moreover, reasoning exactly as in the proof of Lemma \ref{l-T-ben-def} we get also $|g_n(x)|\leq C(1+|x|)^{-\gamma}$ for any $n$. This estimate allow us to pass to the limit in \eqref{g-n-bis} getting that 
%Using \eqref{g-n-bis} we have
\begin{eqnarray}
&& \int_{\R^N}|\na g_n|^2\,dx-\int_{\R^N}\frac {\l}{|x|^2}g_n^2=C(\l)\int_{\R^N}(v_n^+)^{2^*-1}g_n\, dx\to\nonumber\\
&&C(\l)\int_{\R^N}(\bar v^+)^{2^*-1}\bar g\, dx=\int_{\R^N}|\na \bar g|^2\,dx-\int_{\R^N}\frac {\l}{|x|^2}\bar g^2\, dx.\nonumber
\end{eqnarray}
By Lemma \ref{l2} this implies that $g_n\to g$ strongly in $D^{1,2}(\R^N)$. To finish the proof we need to show that $\nor g_n-g\nor_{\gamma}<\e$ if $n$ is large enough. To this end observe that $g_n-g$ weakly solves
$$-\Delta (g_n-\bar g)-\frac{\l}{|x|^2}(g_n-\bar g)=C(\l)\left( (v_n^+)^{2^*-1}- (\bar v^+)^{2^*-1}\right) \, \hbox{ in }\R^N$$
and, since $v_n$ and $\bar v$ are uniformly bounded in $L^{\infty}_{\gamma}(\R^N)$ then, as in Lemma \ref{l-T-ben-def} we have that $|g_n-\bar g|\leq C w$, where $w$ is  defined in \eqref{def-w}.  Then, from \eqref{w-infinito} we have that there exists $R_0>0$ such that  $(1+|x|)^{\gamma}| g_n(x)-\bar g(x)|\leq \frac \e 3$ in $\R^N\setminus B_{R_0}$ uniformly in $n$. Using \eqref{w-in0} instead we get that there exists $r_0$ such that  $ (1+|x|)^{\gamma}| g_n(x)-\bar g(x)|\leq \frac \e 3$ in $ B_{r_0}$ uniformly in $n$. Finally 
since, $v_n\to \bar v$ in $L^{\infty}(B_{R_0}\setminus B_{r_0})$  then we get that $(1+|x|)^{\gamma}|g_n-\bar g|<\frac \e 3$ for $n$ large enough in $B_{R_0}\setminus B_{r_0}$
and the proof of $i)$ is complete.\\[.3cm]
Let us prove $ii)$.
Let us consider the linearized operator of $I-T$ in $(\l,u_{\l})$. We have that
$$<I-T'_v(\l,u_{\l})I,w>=w-P_{K_\l}\left(\left( -\Delta-\frac{\l}{|x|^2}I\right)^{-1}\left(C(\l)\frac{N+2}{N-2}u_{\l}^{2^*-2}w\right)\right)$$
so that $w-T'_v(\l,u_{\l})w=0$ if and only if $w\in K_{\l}\cap X$ satisfies
$$-\D w-\frac{\l}{|x|^2}w -C(\l)\frac{N+2}{N-2}u_{\l}^{2^*-2}w=L C(\l)\frac{N+2}{N-2} u_{\l}^{2^*-2}Z_\l $$
%-\Delta w-\frac{\l}{|x|^2}w-C(\l)\frac{N+2}{N-2}u_{\l}^{2^*-2}w=L C(\l)\frac{N+2}{N-2}u_{\l}^{2^*-2}Z(x)\quad \hbox{ in }\R^N\setminus \{0\}
weakly in $D^{1,2}(\R^N)$ for some $L=L(w)\in\R$. 
Multiplying by $Z_\l$ and recalling the equation satisfied by $Z_\l$ we get 
% $Z_\l$ satisfies
%\begin{equation}\label{equazione-Z}
%\%int_{\R^N}\na Z_\l\cdot \na \psi -\int_{\R^N}\frac{\l}{|x|^2}Z_\l\psi -C(\l)\frac{N+2}{N-2}\int_{\R^N}u_{\l}^{2^*-2}Z_\l \psi=0
%-\Delta Z-\frac{\l}{|x|^2}Z-C(\l)\frac{N+2}{N-2}u_{\l}^{2^*-2}Z=0\quad \hbox{ in }\R^N\setminus \{0\}. 
%\end{equation}
%for any $\psi \in D^{1,2}(\R^N)$.
%Using $w$ as a test function 
%then we get
%$$\int_{\R^N}\na w\cdot \na Z_\l-\int_{\R^N}\frac{\l}{|x|^2}w Z_\l-C(\l)\frac{N+2}{N-2}\int_{\R^N}u_{\l}^{2^*-2}wZ_\l=0$$
%so that
$$0=LC(\l)\frac{N+2}{N-2} \int_{\R^N}u_{\l}^{2^*-2}Z_\l^2$$
and this implies $L=0$. Then $w\in K_{\l}$ is a weak solution of 
\begin{equation}\label{prima}
-\Delta w-\frac{\l}{|x|^2}w-C(\l)\frac{N+2}{N-2}u_{\l}^{2^*-2}w=0\quad \hbox{ in }\R^N.
\end{equation}
Using Lemma \ref{lemma-lin} then we get that if $\l\neq \l_n$  equation \eqref{prima} has only the solution $Z_\l$ which is not in $K_\l$. This means that equation \eqref{prima} has in $K_\l$ only the solution 
 $w=0$ and the operator $I-T'_v(\l,u_{\l})I$ is indeed invertible, concluding the proof.
\end{proof}
\noindent To prove the bifurcation result (Theorem \ref{bif}) we need to exploit some of the symmetries of problem \eqref{1}. So we define  the subspace $\h$  of  $X$ given by 
\begin{equation}\nonumber
\h:=\{v\in  X\, ,\hbox{s.t. }v(x_1,\dots,x_N)=v(g(x_1,\dots,x_{N-1}),x_N)\, ,\hbox{ for any }g\in O(N-1)\}.
\end{equation}
Now let us consider  the subgroups
$\mg_h$ of $O(N)$ defined by
\begin{equation}\nonumber%\label{G_h}
\mg_h=O(h)\times O(N-h)\quad \hbox{ for }1\leq h\leq \left[\frac N2\right]
\end{equation}
where $\left[ a\right]$ stands for the integer part of $a$.
We consider also the subspaces $\h^h$ of $X$ of functions invariant by the
action of $\mg_h$. \\
The results of Smoller and Wasserman in \cite{SW86} and \cite{SW90} implies that, for any $j$ the eigenspace of the Laplace Beltrami operator related to $\mu_j$ (see Section \ref{s3}) contains only one eigenfunction which is $O(N-1)$-invariant (or which is invariant by the action of $\mg_h$). Then, Corollary \ref{cor-1} implies that 
$$m_{\h}(\l_j-\e)-m_{\h}(\l_j+\e)=1$$
if $\e$ is small enough, where $m_{\h}$ denotes the Morse index of $u_\l$ in the space $\h$ (or $\h^h$).\\
The change of the Morse index of $u_\l$ is a good clue to having the bifurcation, but since $u_{\l}$ is radially degenerate we have to use the projection $P_\l$, changing problem \eqref{1} with problem \eqref{molt}.\\
What we can do, at this step, is to prove a bifurcation result for problem \eqref{molt}. To prove this we need ``roughly speaking'' that the Morse index of $u_\l$ as a solution of problem \eqref{molt} is the same as $m(\l)$ and this is proved in the following proposition. 
\begin{proposition}\label{p2}
The number of the eigenvalues of $T'_v(\l,u_\l)$ counted with multiplicity in $(1,+\infty)$ coincides with the morse index $m(\l)$ of $u_\l$.
\end{proposition}
%The Leray Schauder degree of the operator $I-T$ in the space $\h$ changes along the curve $(\l,u_\l)$ as $\l$ goes from $\l_n-\e$ to $\l_n+\e$ and $\e$ small enough. Further if ${\textcolor{red}{N}}$ is even also the Leray Schauder degree of the operator $I-T$ in the space $\h^h$ changes, for every $h=1,\dots, \left[\frac N2\right]$. 
%\end{proposition}
%%%GUARDARE DOPO!!!!!!
\begin{proof} %First observe that for every $\l$ the pair $(\l,u_{\l})$ is a solution of \eqref{1} and {\textcolor{red}{hence}} $u_\l-T(\l,u_\l)=0$ for every $\l\in (-\infty,\frac{(N-2)^2}4)$. Moreover from Proposition \ref{compattezza} we {\textcolor{red}{have}} that the Leray Schauder degree of $I-T$ is well defined {\textcolor{red}{(di solito il Leray Schauder degree e' definito in un insieme limitato?)}}in $D^{1,2}(\R^N)$ and hence also in $\h$ or $\h^h$.\\
%As proved in \cite[Theorem 3.20]{AM} the Leray Schauder degree of $I-T$  in $0$ for $\l\neq \l_n$ is equal to $(-1)^{\gamma}$ where $\g$ is the number of the eigenvalues of $T'_v(\l,u_{\l})$  counted with multiplicity contained in $(1,+\infty)$ {\textcolor{red}{(in realta' in AM c'e' scritto che l'indice vale $(-1)^{\gamma}$. Oppure si parla di grado in domini limitati.. )}}.\\
$\L$ is 
an eigenvalue for the linear operator $T'_v(\l,u_{\l})I$  if and only if 
$$\L I-T'_v(\l,u_{\l})I=0$$
has a nontrivial solution in $X\cap K_\l$. % which belongs to $K_\l$.  
This means that we have  to find $w\in X\cap K_\l$, $w\neq 0$ which verifies
\begin{equation}\label{4.5}
-\Delta w-\frac{\l}{|x|^2}w= \frac 1 {\L} C(\l)\frac{N+2}{N-2}u_{\l}^{2^*-2}w+\frac{L}{\L} C(\l)\frac{N+2}{N-2}u_{\l}^{2^*-2}Z_\l \quad\hbox{in }\R^N
\end{equation}
for some $L=L(w)\in \R$ and $\frac 1{\L}\in (0,1)$. \\% Note that since $T'$ maps $D^{1,2}(\R^N)$ in $K_{\l}$ and $\L w= T'_v(\l,u_{\l}) w$ then $w\in K_{\l}$.   \\
Observe that, since $\L\neq 1$, the function $w_1=\frac{L}{\L-1}Z_\l$ is always a solution of \eqref{4.5} (that does not belong to $K_\l$) and all the other solutions of \eqref{4.5} are given by $w=w_1+\tilde w$ 
where $\tilde w \in X\cap K_\l$ satisfies
\begin{equation}\label{4.6}
-\Delta\tilde w-\frac{\l}{|x|^2}\tilde w= \frac 1 {\L} C(\l)\frac{N+2}{N-2}u_{\l}^{2^*-2} \tilde w\quad\hbox{in }\R^N.
\end{equation}
Now, if $\frac 1{\L}$ is not an eigenvalue of \eqref{4.6} then $\tilde w=0$ and \eqref{4.5} has only the solution $w_1$. But $w_1$  is not in $K_{\l}$
so that $w=0$ and $L=0$ in \eqref{4.5}.\\
If else, $\frac 1{\L}$ is an eigenvalue of \eqref{4.6} and $\tilde w$ a corresponding eigenfunction we can use $\tilde w$ as a test function in \eqref{4.5}, $Z_\l$ as a test function in  \eqref{4.6}, getting that 
$$\int_{\R^N}u_{\l}^{2^*-2}Z_\l \tilde w\, dx=0$$
so that $\tilde w\in K_\l$. Since $w_1\notin K_{\l}$ this implies $L=0$ in \eqref{4.5} so that equation \eqref{4.5} coincides with equation \eqref{4.6}. \\
We have shown so far that the number of the eigenvalues of $T'_v(\l,u_\l)$ counted with multiplicity in $(1,+\infty)$ is equal to the number of the eigenvalues of \eqref{4.6} counted with multiplicity in $(0,1)$, and this is the Morse index of $u_\l$.
\end{proof}
\noindent From Proposition \ref{p2} we have that the number of the eigenvalues of $T'_v(\l,u_\l)$ counted with multiplicity in $(1,+\infty)$ decreases by one going from $\l_j-\e$ to $\l_j+\e$ and $\e$ small enough in the space $\h$ (or $\h^h$) and this is sufficient to have the bifurcation.\\
We do not give the details of the global bifurcation result for problem \eqref{molt}, we only sketch the proof of the local bifurcation result to have an idea how to use the results of Propositions \eqref{compattezza} and \eqref{p2}.\\
Then the global bifurcation result will follow reasoning as in \cite[Theorem 3.3]{G}, (see also \cite{AM}). 
\begin{proposition}
The points $(\l_j,u_{\l_j})$ are nonradial bifurcation points for the curve $(\l,u_\l)$ of radial solutions of \eqref{molt}.
\end{proposition}
\begin{proof}
Assume by contradiction that $(\l_j,u_{\l_j})$ is not a bifurcation point for \eqref{molt}, for some $j$. Then there exists $\e_0>0$ such that $\forall \e\in (0,\e_0)$ and $\forall c\in (0,\e_0)$ 
$$I-T(\l,v)\neq 0$$
for any $\l\in (\l_j-\e,\l_j+\e)\subset(-\infty,0)$ and for any $v\in \h$ (or in $\h^h$) such that $\nor v-u_\l\nor_X\leq c$ and $v\neq u_\l$. \\
Let $\Gamma :=\{(\l,v)\in   (\l_j-\e,\l_j+\e)\times \h\,:\,\nor v-u_\l\nor_X\leq c \}$ and $\Gamma_\l:=\{v\in \h\hbox{ s.t. }(\l,v)\in \Gamma\}$. By the homotopy invariance of the Leray Schauder degree we have that
\begin{equation}\label{constant}
deg\left(I-T(\l,\cdot),\Gamma_\l,0\right) \hbox{ is constant on }(\l_j-\e,\l_j+\e).
\end{equation}
Since the linearized operator is invertible for $\l=\l_j-\e$ and $\l=\l_j+\e$ we can compute the Leray Schader degree and we have that
$$deg\left(I-T(\l_j\pm\e,\cdot),\Gamma_{\l_j\pm\e},0\right)%=index \left(I-T(\l_j\pm\e,\cdot),0\right)
=(-1)^{\b(\l_j\pm\e)}$$
where $\b(\l) $  
is the number of the eigenvalues of $T'_v(\l,u_{\l})$  counted with multiplicity contained in $(1,+\infty)$, see \cite[Theorem 3.20]{AM}. Then Proposition \ref{p2}  implies that 
$$deg\left(I-T(\l_j-\e,\cdot),\Gamma_{\l_j-\e},0\right)=-deg\left(I-T(\l_j+\e,\cdot),\Gamma_{\l_j+\e},0\right)$$
contradicting \eqref{constant}. Then $(\l_j,u_{\l_j})$ is a bifurcation point for \eqref{molt} and the bifurcating solutions are nonradial since $u_\l$ is radially nondegenerate in $K_\l$.
\end{proof}

%%%CONTROLLARE DOPO
% Leray Schauder degree {\textcolor{red}{(oppure l'indice..)}} of $I-T$ in $0$ for $\l\neq \l_n$ is equal to the number of the eigenvalues $\frac 1{\L}$ contained in $(0,1)$ of \eqref{4.6} counted with multiplicity and this coincide with the number of the eigenvalues $\L$ of problem \eqref{aut} which are less than zero, i.e. the Morse index of $u_\l$.\\
%Now we consider the case of the space $\h$.  
%As proved in Corollary \ref{cor-1} we have that the Morse index of $u_{\l}$ changes as $\l$ crosses the values $\l_n$. By Smoller and Wasserman results, see \cite{SW86} or \cite{SW90}, we know that problem \eqref{aut} has a unique eigenfunction in $\h$ (or in $\h^h$ if $n$ is even) corresponding to $\l_n$, so that the Morse index of $u_\l$ in $ \h$ (or in $\h^h$) going from $\l_n-\e$ to $\l_n+\e$ for $\e$ small enough changes by one. This implies that the Leray Schauder degree of $I-T$ in $\h$ goes from $-1$ to $1$ or vice-versa and this ends the proof. 
%\end{proof}
\noindent Finally we can state the global bifurcation result for \eqref{molt}.
\begin{proposition}\label{bif-molt}
Let us fix $j\in \N$ and let $\l_j$ be as defined in \eqref{lambda-j}. Then\\
i) If $j$ is odd there exists at least a continuum of nonradial solutions to \eqref{molt}, invariant with respect to $O(N-1)$, bifurcating from the pair $(\l_j,u_{\l_j})$.\\
ii) If $j$ is even there exist at least $\big[\frac N2\big]$ continua of nonradial solutions to
\eqref{molt} bifurcating from the pair $(\l_j,u_{\l_j})$. The first branch is $O(N-1)$ invariant, the second is $O(N-2)\times O(2)$ invariant, etc.
%Moreover all these solutions belong to $X$.
\end{proposition}
\noindent The final step for the proof of Theorem \ref{bif} will be to show that the solutions we have found in Theorem \ref{bif-molt} are indeed solutions of \eqref{1}. This will be done in the next section.

\subsection{The Lagrange multiplier is zero }\label{s5} 
In the previous section we proved the existence of solutions $u_{\e,n}$ and parameters $\l_{\e,n}, L_\e$ verifying
\begin{eqnarray}\label{u-epsilon}
%&&\int_{\R^N}\na u_{\e,n}\cdot \na \psi -\int_{\R^N}\frac{\l_{\e,n}}{|x|^2} u_{\e,n}\psi- C( \l_{\e,n} )\int_{\R^N}u_{\e,n}^{2^*-1} \psi\\
%&&=L_\e C(\l_{\e,n})\frac{N+2}{N-2}\int_{\R^N}u_{{\e,n}}^{2^*-2}Z_{\e,n}\psi\nonumber
\!\!\!\!\!\!\!-\Delta u_{\e,n}-\frac{\l_{\e,n}}{|x|^2}u_{\e,n} -C(\l_{\e,n})u_{\e,n}^{2^*-1}\!=\!L_\e C(\l_{\e,n})\frac{N+2}{N-2}u_{{\e,n}}^{2^*-2}Z_{\e,n} \ \hbox{in }\R^N
\end{eqnarray} 
where $Z_{\e,n}=Z_{\l_{\e,n}}$, 
with $\l_{\e,n}$ and $u_{\e,n}$ and $L_\e$ such that $\l_{0,n}=\l_n$ and $u_{\e,n}=u_{\l_n}$ and $L_0=0$. \\
In the following we denote by $C$ a generic constant (independent of $n$ and $\e$) which
can change from line to line.\\
First we prove a bound on $L_\e$.
\begin{lemma}
We have
$$|L_\e|\leq C.$$
\end{lemma}
\begin{proof}
Using $Z_{\e,n}$ as a test function in \eqref{u-epsilon} we get
\begin{align}
& L_\e C(\l_{\e,n})\frac{N+2}{N-2}\int_{\R^N}  u_{\l_{\e,n}}^{2^*-2}Z_{\e,n}^2=  \int_{\R^N}\na u_{\e,n}\cdot \na Z_{\e,n} \nonumber\\
&-\int_{\R^N} \frac{\l_{\e,n}}{|x|^2}u_{\e,n} Z_{\e,n}  -C(\l_{\e,n})\int_{\R^N}u_{\e,n}^{2^*-1} Z_{\e,n} . \nonumber
\end{align}
Using Lemma \ref{l2}, the Holder and the Sobolev inequality we get
$$ L_\e C(\l_{\e,n})\frac{N+2}{N-2}\int_{\R^N}  u_{\l_{\e,n}}^{2^*-2}Z_{\e,n}^2\leq C\nor u_{\e,n}\nor_{1,2}\nor Z_{\e,n}\nor _{1,2}$$
so that the claim follows.
\end{proof}
\begin{proposition}
Let $u_{\e,n}$ be the solution of \eqref{u-epsilon}. Then $L_\e=0$ in  \eqref{u-epsilon} for $\e$ small enough.
\end{proposition}
\begin{proof}
Applying the Pohozaev identity \eqref{pohozaev} with $f(x,u)=\frac{\l_{\e,n}}{|x|^2} u+C(\l_{\e,n})u^{2^*-1}+L_\e C(\l_{\e,n})\frac{N+2}{N-2}u^{2^*-2}Z_{\e,n}$
we get
\begin{eqnarray}
&&\int_{\R^N}|\na u_{\e,n}|^2\, -\frac{N}{N-2}\int_{\R^N}\frac{\l_{\e,n}}{|x|^2}u_{\e,n}^2\,-C(\l_{\e,n})\int_{\R^N}u_{\e,n}^{2^*}\, \nonumber\\
&&-\frac{2N}{N-2}L_\e C(\l_{\e,n})\int_{\R^N}u_{\e,n}^{2^*-1}Z_{\e,n}
+\frac 2{N-2}\int_{\R^N}\frac{\l_{\e,n}}{|x|^2}u_{\e,n}^2\nonumber\\
&&-\frac 2{N-2}L_\e C(\l_{\e,n})\int_{\R^N}u_{\e,n}^{2^*-1}x\cdot \na Z_{\e,n}=0\nonumber.
\end{eqnarray}
Using $u_{\e,n}$ as a test function in \eqref{u-epsilon} then we get
$$L_\e\int_{\R^N} \left(u_{\e,n}^{2^*-1}Z_{\e,n}+u_{\e,n}^{2^*-1}x\cdot \na Z_{\e,n}\right)=0$$
and this implies $L_\e=0$ if we show that the integral is different from zero. Let us recall that $u_{\e,n}\to u_{\l_n}$, $ Z_{\e,n}\to Z_{\l_n}$ as $\e\to 0$. Moreover, by the definition of $Z_{\e,n}$ and since $\l_n<0$ we get $Z_{\e,n}=O\left((1+|x|)^{2-N}\right)$ and  $|\nabla Z_{\e,n}|=O\left((1+|x|)^{1-N}\right)$. Finally since $u_{\e,n}\in X$ we have that $u_{\e,n}\in X$ and then $u_{\e,n}=O\left((1+|x|)^{\gamma}\right)$ with $\frac{N-2}2<\gamma<N-2$. So by the dominate convergence theorem we derive that
$$ \int_{\R^N} \left(u_{\e,n}^{2^*-1}Z_{\e,n}+u_{\e,n}^{2^*-1}x\cdot \na Z_{\e,n}\right)\to \int_{\R^N}u_{\l_n}^{2^*-1}x\cdot \na Z_{\l_n}\neq 0$$
so that $L_\e=0$ if $\e$ small enough, concluding the proof.
\end{proof}

\sezione{The subcritical case $1<p<\frac{N+2}{N-2}$}\label{s6} 
Let us start this section recalling some known facts. % Set $B_1=\left\{x\in\R^N:|x|\le1\right\}$ 
%and $\l_1$ be the first eigenvalue of $-\D-\l$ in $B_1$ with zero Dirichlet boundary condition. 
%Let us denote by 
%$$H=\left\{u\in H^1((0,1),r^{M-1}dr)\hbox{ such that }u(1)=0\right\}.$$
%and we recall that $u_\l$ is non degenerate if the linear problem
%\begin{equation}\label{lin-3}
%\left\{\begin{array}{ll}
%-\Delta w-\frac {\l}{|x|^2}w=pu_\l^{p-1}w & \hbox{ in }B_1\\
%w=0& \hbox{ on }\partial B_1\\
%w\in H^1_0(B_1)
%\end{array}\right.
%\end{equation}
%admits only the trivial solution. 
Next theorem collects some results of different authors (see \cite{GNN} and \cite{S}).
\begin{theorem}\label{d1}
Let $1<p<\frac{L+2}{L-2}$ and $v_L$ be the unique positive solution of 
\begin{equation}\label{palla}
\left\{\begin{array}{ll}
-v''-\frac{L-1}rv'=v^p& \hbox{ in }(0,1)\\
v>0&\hbox{ in }(0,1)\\
v'(0)=v(1)=0
\end{array}\right.
\end{equation}
where $L$ is a real number greater than $1$.
Then $v_L$ is nondegenerate and its Morse index is $1$.
\end{theorem}
\begin{remark} {\rm
Theorem \ref{d1} allows to establish the existence of the branch of radial solutions $u_\l$ as stated in Theorem \ref{d2} in the Introduction. Moreover, using the transformation $\mathcal{L}_p$
we are able to find the behaviour of the radial solution $u_\l$ at zero, see \eqref{origin}.}
\end{remark} 
\begin{remark}  {\rm As $\l>0$ problem \eqref{i10}  has been studied in \cite{CG}. In this case the authors proved the existence of a unique radial solution $u_\l$ and his behaviour near the origin, which is exactly the same as in \eqref{origin}.
Their proof relies on the the moving plane method, which ensures that every positive solution is radial, and on the phase plane analysis of the radial solutions. Both steps strongly rely on the hypothesis that $\l>0$ and cannot be extended to $\l\leq 0$. Using the map $\mathcal{L}_p$ we easily obtain a new proof of the results of \cite{CG} and we extend them to the case $\l<0$.}
\end{remark} 
\begin{remark}\label{R1}{\rm
The nondegeneracy result in Theorem \ref{d1} and the implicit function theorem imply that the function $\l\rightarrow v_\l$ is $C^1$.}
\end{remark}

\begin{proof}[Proof of Theorem \ref{d2}]
Let $u(r)$ be a radial solution to \eqref{i10} and let $v(r)=\mathcal{L}_p(u(r))$ as defined in \eqref{i1}.
From Proposition \ref{i2} we know that the transformed function $v(r)$ satisfies
\begin{equation}\label{eqv}
\left\{\begin{array}{ll}
-v''-\frac{M-1}rv'=A(\l,p)v^p & \hbox{ in }(0,1)\\
v>0& \hbox{ in }(0,1)\\
v(1)=0
\end{array}\right.
\end{equation}
with  $M$ as in \eqref{i6} and $A(\l,p)$ as in \eqref{i5}.\par
Moreover, a straightforword 
computation shows that if $1<p<\frac{N+2}{N-2}$ then $1<p<\frac{M+2}{M-2}$.
Then, by \eqref{i2b} we have that 
\begin{equation}\label{energia}
\int_0^{1}r^{M-1}\left(v'(r)\right)^2=\frac 1{\nu_\l}\int_0^{1}s^{N-1}
\left(u'(s)^2-\frac\l{s^2}u^2(s)\right)\leq C.
\end{equation}
We want to use  \eqref{energia} to prove that the function $v$ satisfies $v'(0)=0$ also. This will imply the existence and uniqueness result using
Theorem \ref{d1} with $L=M$.\\
To this end we let $\widetilde v(r)=\frac 1{r^ {M-2}}v\left( \frac 1r\right)$. The function $\widetilde v$ solves the equation 
\begin{equation}\label{trasf-kelv}
\left\{\begin{array}{ll}
-(r^{M-1}\widetilde v'(r))'= A(\l,p) r^{(M-2)p-3}\widetilde v^p(r) & \hbox{ in }(1,+\infty)\\
\widetilde v(1)=0
\end{array}\right.
\end{equation}
and satisfies $$\int_1^{+\infty}r^{M-1}\left(\widetilde v'(r)\right)^2\leq C.$$
Reasoning exactly as in the radial Lemma of Ni (see \cite{Ni}) then we get that 
\begin{equation}\label{stima-nb}
\widetilde v(r)\leq Cr^{\frac {2-M}2}
\end{equation} 
so that $\widetilde v(r)\to 0$ as $r\to +\infty$ since $M>2$. Let $r_0$ be a maximum point for $\widetilde v$ in $(1,+\infty)$. Integrating \eqref{trasf-kelv} in $(r_0,r)$ then we get
$$-r^{M-1}\widetilde v'(r)=A(\l,p)\int_{r_0}^r s^{(M-2)p-3}\widetilde v^p(s)\, ds.$$
Using estimate \eqref{stima-nb} we have 
\begin{eqnarray}
&&\Big| \int_{r_0}^r s^{(M-2)p-3}\widetilde v^p(s)\, ds\Big|\leq C \Big|  \int_{r_0}^r s^{\frac {M-2}2 p-3}\,ds\Big|.\label{stima-nb2}
\end{eqnarray}
This implies that 
\begin{equation}\label{stima-meglio}
|\widetilde v'(r)|\leq 
\left\{\begin{array}{ll}
Cr^{1-M} & \hbox{ when } p<\frac 4{M-2}\\
Cr^{1-M}\log r & \hbox{ when } p=\frac 4{M-2}\\
Cr^{-1-M+\frac{M-2}2 p}& \hbox{ when } p>\frac 4{M-2}
\end{array}\right.
\end{equation}
When $p<\frac 4{M-2}$  \eqref{stima-meglio} produces the optimal decay for $\widetilde v'(r)$. Otherwise, if $p\geq \frac 4{M-2}$ we need to repeat the procedure estimating again the integral in \eqref{stima-nb2} using \eqref{stima-meglio}. In any case after a finite number of steps we get that
\begin{equation}\label{stima-der-v}
|\widetilde v'(r)|\leq Cr^{1-M}
\end{equation}
and this implies that 
\begin{equation}\label{stima-v}
\widetilde v(r)\leq Cr^{2-M}.
\end{equation}
Turning back to the function $v$ then we get that 
$$v(r)\leq C \hbox{ in }[0,1].$$
Further, using the definition of $\widetilde v$ and estimates
\eqref{stima-der-v} and \eqref{stima-v} we have 
$$\lim_{r\to 0}r^{M-1}v'(r)=\lim_{r\to 0}-\frac 1r\widetilde v'\left(\frac 1r\right)+(2-M)\widetilde v(\frac 1r)=0$$
since $M>2$. Integrating equation \eqref{eqv} then we obtain that 
\begin{equation}\label{eq-int}
-r^{M-1}v'(r)=A(\l,p)\int_0^r s^{M-1}v^p(s)\, ds
\end{equation}
so that $v'(r)<0$ in $(0,1)$ and $\lim_{r\to 0}v(r)$ exists and it is finite showing that $v$ is continuous at the origin.
Using \eqref{eq-int} again we have
$$\lim_{r\to 0}v'(r)=-A(\l,p)\lim_{r\to 0}\frac {\int_0^r s^{M-1}v^p(s)\, ds}{r^{M-1}}=-A(\l,p)\lim_{r\to 0}\frac{r}{M-1}v^p(r)=0.$$
This shows that the transformed function $v(r)$ has to be a solution of \eqref{palla} since the constant $A(\l,p)$ can be merged into the equation. Theorem \ref{d1} then implies the existence and uniqueness of the radial solution. The final estimate follows inverting the transformation $\mathcal{L}_p$ and using the continuity of $v(r)$ in $0$.
\end{proof}
\begin{corollary}\label{C2}
We have that the radial solution $u_\l$ to  \eqref{i10} satisfies $u_\l(0)=0$ if $\l<0$.
\end{corollary}
\begin{proof}
It is enough to remark that $\nu_\l>1$ as $\l<0$. Then the claim follows by \eqref{origin}.
\end{proof}
In the rest of the section we will denote by $u_\l$ the unique radial solution to \eqref{i10} and by 
$$H=\left\{u\in H^1((0,1),r^{M-1}dr)\hbox{ such that }u(1)=0\right\}.$$
Set $v_\l(r)$ as in \eqref{i1} and $\L(\l)$ be as defined in \eqref{i12}. Although the embedding of $H\hookrightarrow L^2((0,1),r^{M-3}dr)$ is not compact, $\L(\l)$ is achieved. This is a consequence of Proposition \ref{d12-bis} in the Appendix, whose proof is basically the same of Proposition A.1 in \cite{GGN}. %We repeat it for the reader convenience.
Then we have the following result:
\begin{corollary}\label{v4}
The first eigenvalue $\L(\l)$ defined in \eqref{i12} is achieved.
\end{corollary}
\begin{proof}
Since  $\L(\l)\le(1-p)\frac{\int_0^1v_\l^{p+1}r^{M-1}}{\int_0^1v_\l^2r^{M-3}}<0$  the claim follows by Proposition \ref{d12-bis}.
\end{proof}
%Let us prove an auxiliary lemma which will be useful in the sequel.
As in the previous section we study the linearized operator at the solution $u_\l$ and 
we recall that $u_\l$ is non degenerate if the linear problem
\begin{equation}\label{lin-3}
\left\{\begin{array}{ll}
-\Delta w-\frac {\l}{|x|^2}w=pu_\l^{p-1}w & \hbox{ in }B_1\\
%w=0& \hbox{ on }\partial B_1\\
w\in H^1_0(B_1)
\end{array}\right.
\end{equation}
admits only the trivial solution. 
\begin{theorem}\label{d4}
Let $k\in \N$, $k\geq 1$ and $\l\leq \frac{(N-2)^2}4$. The linearized equation at the radial solution $u_\l$, i.e. equation \eqref{lin-3}, admits a solution if and only if $\l$ satisfies
\begin{equation}\label{d5}
-\L(\l)=\frac{16k(N-2+k)}{\left[(p-1)\left(2-N+\sqrt{(N-2)^2-4\l}\right)+4\right]^2},
\end{equation}
for some $k\geq 1$. Moreover 
%such that 
the space of solutions of \eqref{lin-3} corresponding to a value of $\l$ which satisfies \eqref{d5} related to some $k$, has dimension $\frac{(N+2k-2)(N+k-3)!}{(N-2)!\,k!} $ and it is spanned by 
\begin{equation} \nonumber\label{d6}
Z_{k,i,\l}(x)=\frac1{|x|^\frac{a}b}\psi_1\left(|x|^\frac1b\right)Y_{k,i}(x)
\end{equation}
where $\psi_1$ is the positive eigenfunction associated to $\L(\l)$ and $\{Y_{k,i}\}$,\\ ${i=1,\dots,\frac{(N+2k-2)(N+k-3)!}{(N-2)!k!}}$, form a basis of $\mathbb{Y}_k(\R^N)$, the space of all homogeneous harmonic polynomials of degree $k$ in $\R^N$.\\
Finally, for every $k\geq 1$ there exists at least one value of $\l$ that satisfies \eqref{d5} and if $\l$ is not a solution to \eqref{d5} then the solution $u_\l$ is nondegenerate.
\end{theorem}
\begin{proof}
The beginning of the proof is basically the same as Lemma \ref{lemma-lin}. Let $v$ be a solution to \eqref{lin-3} and decomposing as sum of spherical harmonics we reduce ourselves to study the following ODE,
\begin{align*}% \label{d7}
\begin{cases}
-\psi_k''(r) - \frac{N-1}{r}\psi_k'(r) + \frac{\mu_k-\l}{r^2}\psi_k(r) = pu_\l^{p-1}(r)\psi_k(r)&
\text{in} \,\,(0,1) \\
%\psi'_k(0)=0 \hbox{ if }k=0\hbox{ and }\psi_k(0)=0 \hbox{ if }k\geq 1,\\
\psi_k(1)=0 \ , \
\int_0^1r^{N-1}\left(\psi_k'(r)\right)^2\, dr<\infty
\end{cases}
\end{align*}
where $\mu_k=k(N-2+k)$. Setting again,
$\widehat \psi_k(r)=r^a\psi_k\left(r^b\right)$
we have that $\widehat \psi_k$ solves
\begin{align} \label{d8}
\begin{cases}
-\widehat \psi_k''(r) - \frac{M-1}{r}\widehat \psi_k'(r) + \frac{b^2\mu_k}{r^2}\widehat \psi_k(r) = pv_\l^{p-1}(r)\widehat\psi_k&
\text{in} \,\,(0,1) \\
\widehat\psi_k(1)=0 \ , \ \widehat\psi_k\in H
\end{cases}
\end{align}
Note that, since $v_\l$ is nondegenerate, from Theorem \ref{d1}, the previous problem cannot have solutions for $k=0$. So we assume that $k\ge1$.\\
By Theorem  \ref{d1}  we get that \eqref{d8} has a nontrivial solution belonging to the space $H$ if and only if $-b^2\mu_k=\L(\l)$ which is the unique negative eigenvalue. Moreover by Lemma \ref{B2} we get that $\widehat \psi_k\in L^\infty(0,1)$. Recalling \eqref{alpha} we get that equation \eqref{lin-3} admits a solution if and only if 
\begin{equation}\label{d9}
-\L(\l)=\frac{16k(N-2+k)}{\left[(p-1)\left(2-N+\sqrt{(N-2)^2-4\l}\right)+4\right]^2}
\end{equation}
for some $k\geq 1$.
Since the solution $u_\l$ is not explicitly known as in the previous section, we have to show that \eqref{d9} has at least a solutions. Let us consider the two {\em{limit}} cases $\l=0$ and  $\l=-\infty$. Note that by Remark \ref{R1} we derive that $\L(\l)$ is a continuous function of $\l$.
\\
{\em Case i)} $\l=0$.\par
First let us study the limit of the solution $v_\l$ to \eqref{i11}  as $\l$ goes to zero.  By the uniqueness result of Theorem \ref{d1} we have that $v_\l$ can be characterized as
\begin{equation} \nonumber
\inf\limits_{\int_0^1v(r)^{p+1}r^{M-1}=1}\int_0^1\left(v'(r)\right)^2r^{M-1}\quad v\in H.
\end{equation}
Then it is easy to see that $v_\l$ achieves this infimum and then $\int_0^1\left(v_\l'(r)\right)^2r^{M-1}\le C$ for some positive constant $C$ independent of $\l$. So $v_\l\rightharpoonup v_0$ where, from Remark \ref{R1}, $v_0$ satisfies
\begin{equation}\nonumber
\left\{\begin{array}{ll}
-v''-\frac{N-1}rv'=v^p & \hbox{ in }(0,1)\\
v>0& \hbox{ in }(0,1)\\
v'(0)=v(1)=0,
\end{array}\right.
\end{equation}
since $A(0,p)=1$. Then we get, for any $k\ge1$,
\begin{eqnarray}\label{d10}
&&\lim\limits_{\l\rightarrow0}\left(\L(\l)+\frac{16k(N-2+k)}{\left[(p-1)\left(2-N+\sqrt{(N-2)^2-4\l}\right)+4\right]^2}\right)=\nonumber \\
&&\inf\limits_{v\in H}\frac{\int_0^1\left(|v'|^2-pv_0^{p-1}v^2
\right)r^{N-1}}{\int_0^1v^2r^{N-3}}+k(N-2+k)>0\nonumber
\end{eqnarray}
because $\inf\limits_{v\in H}\frac{\int_0^1\left(|v'|^2-pv_0^{p-1}v^2
\right)r^{N-1}}{\int_0^1v^2r^{N-3}}>1-N$, comparing the eigenfunction which achieves $\L(0)$ with $v'_0$ and using the maximum principle.\\
{\em Case ii)} $\l=-\infty$.\par
For any $k\ge1$, testing $\L(\l)$ with $v_\l$ we get
\begin{equation}\label{d11}
\begin{split}
&\L(\l)+\frac{16k(N-2+k)}{\left[(p-1)\left(2-N+\sqrt{(N-2)^2-4\l}\right)+4\right]^2}\le\\
&(1-p)A(\l,p)\frac{\int_0^1v_\l^{p+1}r^{M-1}}{\int_0^1v_\l^2r^{M-3}}+o(1)
\end{split}
\end{equation}
for $\l$ large enough.
By \eqref{eqv} we have that
\begin{equation}\nonumber\label{d12}
\int_0^1|v_\l'|^2r^{M-1}=A(\l,p)\int_0^1v_\l^{p+1}r^{M-1}
\end{equation}
and using the Hardy inequality for radial function (see \cite{GP}),
\begin{equation}\nonumber\label{d13*}
\int_0^1v^2r^{M-3}\le\left(\frac{M-2}2\right)^2\int_0^1|v'|^2r^{M-1}
\end{equation}
we get that \eqref{d11} becomes
\begin{equation}\nonumber\label{d14}
\begin{split}
&\L(\l)+\frac{16k(N-2+k)}{\left[(p-1)\left(2-N+\sqrt{(N-2)^2-4\l}\right)+4\right]^2}\le\\
&\frac{1-p}{\left(\frac{M-2}2\right)^2}+o(1)=\frac{1-p}{\left(\frac{2}{p-1}\right)^2}+o(1)<0\quad\hbox{for $\l$ large enough.}
\end{split}
\end{equation}
By Cases i) and ii) we derive that, for any $k\ge1$, there exists at least one value of $\l$ which solves \eqref{d9}. 
\end{proof}

\begin{corollary}\label{c4.8}
The Morse index $m(\l)$ of $u_{\l}$ is equal to
$$m(\l)=
\sum_{0 \leq  j<\frac{2-N}2+\frac1 2 \sqrt{(N-2)^2-4\frac {\L(\l)}{b^2}}
\atop_{j\  integer}} \frac{(N+2j-2)(N+j-3)!}{(N-2)!\,j!}.$$
In particular, we have that $m(\l) \to+\infty$ as $\l \to-\infty.$
\end{corollary}
\begin{proof}
Reasoning exactly as in the proof of Proposition \ref{cor-1} we consider the eigenvalue problem with weight and we call $\Gamma_i$ the corresponding eigenvalues. Then, we have that the linearized equation has a negative eigenvalue with weight $\Gamma_i$ if and only if 
$$\L(\l)=b^2\left(\Gamma_i-\mu_j\right)$$
for some $j\in \N$. So we have that the indexes $j$ which contribute to the Morse index of the solution $u_\l$ are those that satisfies
\begin{equation}\label{gamma_i}
\Gamma_i=\frac{\L(\l)}{b^2}+\mu_j<0
\end{equation}
for some $j\in \N$. This implies, recalling the value of $\mu_j$ that $j<\frac{2-N}2+\frac1 2 \sqrt{(N-2)^2-4\frac {\L(\l)}{b^2}}$. The claim follows recalling the dimension of the eigenspace of the Laplace-Beltrami operator related to $\mu_j$. The last assertion follows since $\L(\l)\to -\infty$ for $\l\to -\infty$.
\end{proof}
From Theorem \ref{d4} and Corollary \ref{c4.8} we have that if $\l^*$  satisfies \eqref{d5} and the function $\L(\l)+b^2\mu_k$ changes sign at the endpoints of a suitable  interval containing $\l^*$, then the Morse index of the radial solution $u_\l$ changes. This change in the Morse index is responsible of the bifurcation. From the continuity of $\L(\l)$ we know that there should exists at least one value $\l_k$ that satisfies \eqref{d5} for every $k\geq 1$ but since we do not know if the function $\L(\l)$ is analytic we cannot say that these values $\l_k$ are isolated. To overcame this problem, in the next Proposition we construct an interval $I_k=[\a_k,\b_k]$ which contains at least one of the points $\l_k$ that satisfies \eqref{d5} and at which the function $\L(\l)+b^2\mu_k$ changes sign, and such that the Morse index of the radial solution $u_\l$ at the value $\a_k$ and $\b_k$ differs from $\frac{(N+2k-2)(N+k-3)!}{(N-2)!\,k!} $ which is the dimension of the eigenspace of the Laplace Beltrami operator related to the eigenvalue $\mu_k$.  
\begin{proposition}\label{p4.6}
There exists a sequence $\l_k$ verifying
\begin{equation}
-\L(\l_k)=\frac{16k(N+k-2)}{\left[(p-1)\left(2-N+\sqrt{(N-2)^2-4\l_k}\right)+4\right]^2},
\end{equation}
and a sequence of intervals $I_k=[\a_k,\b_k]\subset(-\infty,0)$ with $\l_k\in I_k$ such that
\begin{equation}\label{f0a}
\L(\b_k)>-\frac{16k(N+k-2)}
{\left[(p-1)\left(2-N+\sqrt{(N-2)^2-4(\b_k)}\right)+4\right]^2},
\end{equation}
\begin{equation}\label{f0b}
\L(\a_k)<-\frac{16k(N+k-2)}{\left[(p-1)\left(2-N+\sqrt{(N-2)^2-4(\a_k}\right)+4\right]^2},
\end{equation}
and
\begin{equation}\label{f0c}
\L(\b_k)<-\frac{16h(N-2+h)}{\left[(p-1)\left(2-N+\sqrt{(N-2)^2-4\b_k}\right)+4\right]^2},
\end{equation}
for any $h< k$ while 
\begin{equation}\label{f0c*}
\L(\a_k)>-\frac{16j(N-2+j)}{\left[(p-1)\left(2-N+\sqrt{(N-2)^2-4\a_k}\right)+4\right]^2},
\end{equation}
for any $j>k$.
%$\l\in[\l_k-\d_{1,k},\l_k+\d_{2,k}]$
\end{proposition}
\begin{proof}
In order to simplify the notation we consider first the case $k=1$.
Set, for $\l\le0$,
\begin{equation}\nonumber
L(\l)=\L(\l)\left[(p-1)\left(2-N+\sqrt{(N-2)^2-4\l}\right)+4\right]^2
\end{equation}
and define $\l_1$ as
\begin{equation}\nonumber
\l_1=\sup\limits_{\l\le0}I_{1,\l}
\end{equation}
where
\begin{equation}\nonumber
I_{1,\l}=\left\{\l\le0\hbox{ such that }L(\l)=-16(N-1)\right\}.
\end{equation}
By cases i) and ii) in Theorem \ref{d4} we get that $I_{1,\l}\ne\emptyset$ and since $L$ is a continuous function we have that 
%$I_{1,\l}$ is a closed set contained in $(-\infty,0)$. So 
there exists $\l_1$ such that
\begin{equation}\nonumber
L(\l_1)=-16(N-1),
\end{equation}
and any other point $\l^*\ne\l_1$ which satisfies
\begin{equation}\nonumber
L(\l^*)=-16(N-1),
\end{equation}
must verify
\begin{equation}\nonumber
\l^*<\l_1.
\end{equation}
Analogously we define, for $k\ge2$
\begin{equation}\label{f1}
\l_k=\sup\limits_{\l\le0}I_{k,\l}
\end{equation}
where
\begin{equation}\nonumber
I_{k,\l}=\left\{\l\le0\hbox{ such that }L(\l)=-16k(N-2+k)\right\}.
\end{equation}
As in the previous case, using the proof of Theorem \ref{d4} we get that there exists $\l_k$ such that
\begin{equation}\nonumber
L(\l_k)=-16k(N-2+k),
\end{equation}
and $\l_k$ achieves \eqref{f1}.\par
Let us show that
\begin{equation}\nonumber
\l_{k+1}<\l_k\quad\hbox{for any }k\ge1.
\end{equation}
Since the function $16k(N-2+k)$ is strictly increasing in $k$ we cannot have that $\l_{k+1}=\l_k$. So by contradiction let us suppose that
\begin{equation}\nonumber
\l_{k+1}>\l_k,
\end{equation}
for some $k\ge1$. Then,
\begin{equation}\label{f2*}
L(\l_{k+1})=-16(k+1)(N-1+k)<-16k(N-2+k)
\end{equation}
and by case i) of the proof of Theorem \ref{d4} we have
\begin{equation}\nonumber
\begin{split}
&\lim\limits_{\l\rightarrow0}L(\l)=
\lim\limits_{\l\rightarrow0}\L(\l)\left[(p-1)\left(2-N+\sqrt{(N-2)^2-4\l}\right)+4\right]^2=\\
&(16+o(1))\lim\limits_{\l\rightarrow0}\L(\l)>-16(N-1)\ge-16k(N-2+k)
\end{split}
\end{equation}
by the intermediate value Theorem for continuous function we get that there exists $\tilde\l_k\ge\l_{k+1}$ such that
\begin{equation}\nonumber
L(\tilde\l_k)=-16k(N-2+k)
\end{equation}
and this is a contradiction with the definition of $\l_k$.\par
So we have shown that
\begin{equation}\nonumber
0>\l_1>\l_2>..>\l_k>..
\end{equation}
Now we prove the claim: again by case i) of Theorem 4.6, since
\begin{equation}\nonumber
\lim\limits_{\l\rightarrow0}L(\l)>-16(N-1)
\end{equation}
 we get that there exists $\l_1<\b_1<0$ such that
\begin{equation}\nonumber
L(\b_1)>-16(N-1)
\end{equation}
%Set $\d_{2,1}=\s-\l_1$ and so $L(\l_1+\d_{2,1})>-16(N-1)$ 
and this implies $\L(\b_1)>-\frac{16(N-1)} {\left[(p-1)\left(2-N+\sqrt{(N-2)^2-4(\b_1)}\right)+4\right]^2}$.  This proves \eqref{f0a}.\par 
On the  other hand, since by \eqref{f2*} we have that  $L(\l_2)=-32N<-16(N-1)=L(\l_1)$  then there exists $\l_2<\a_1<\l_1$ such that $L(\a_1)<-16(N-1)$, 
%Setting $\d_{1,1}=\l_1-\tau$ we get  $L(\l_1-\d_{1,1})<-16(N-1)$ 
which implies $\L(\a_1)<-\frac{16(N-1)}
{\left[(p-1)\left(2-N+\sqrt{(N-2)^2-4(\a_1)}\right)+4\right]^2}.$  This proves \eqref{f0b}.\par
Finally since %$16k(N-2+k)=L(\l_k)$ and 
$\sup\limits_{k\ge2}\l_k=\l_2<\a_1$ we have that
$L(\a_1)>-32N\geq -16j(N-2+j)$ for any $j> 1$ so that \eqref{f0c*} follows.\\
Now we explain how to pass from $k=1$ to $k=2$. 
We take $\b_2=\a_1\in (\l_2,\l_1)$. Then %$\l_2+\d_{2,2}=\l_1-\d_{1,1}$ and then, 
from \eqref{f0b} and \eqref{f0c*} we have that
$$L(\b_2)=L(\a_1)>-32N$$
and
$$L(\b_2)=L(\a_1)<-16(N-2)$$
so that \eqref{f0a} and \eqref{f0c} follows for $k=2$.\\ 
From \eqref{f2*} we have that $L(\l_3)=-48(N+1)<-32N=L(\l_2)$. Then there exists $\l_3<\a_2<\l_2$ such that $-48(N+1)<L(\a_2)<-32N$ so that $L(\a_2)<-32N$ and this proves \eqref{f0b} for $k=2$. Finally by the choice of $\a_2$ we have $L(\a_2)>-48(N+1)\geq -16j(N-2+j)$ for any $j>2$ so that \eqref{f0c*} is proved for $k=2$. The general case can be carried out with the same proof. 
\end{proof}

As in Section \ref{s4} one can define the operator $T(\l,v)\,:\,(-\infty,0)\times H^1_0(B_1)\cap L^{\infty}(B_1)\longrightarrow H^1_0(B_1)\cap L^{\infty}(B_1)$ as $T(\l,v)=\left(-\Delta-\frac{\l}{|x|^2}I\right)^{-1}\left((v^+)^p\right)$ and look for zeros of $I-T(\l,v)$. Letting $X=H^1_0(B_1)\cap L^{\infty}(B_1)$ and reasoning as in the proof of Lemma \ref{l-T-ben-def}, we have that the operator $T$ is well defined from $(-\infty,0)\times X$ into $X$. $T$ is continous with respect to $\l$ and it is compact from $X$ into $X$ for any $\l\in(-\infty,0)$ fixed.\\
Moreover the linearized operator $I-T'(\l,u_{\l})I$ is invertible for any value of $\l$ which do not satisfy \eqref{d5}.\\
To prove the bifurcation we have to consider as in the previous section the suspace $\mathcal{H}$ of $X$ of functions which are $O(N-1)$-invariant and the subspaces $\mathcal{H}^h$ of $X$ of functions which are invariant by the action of $\mathcal{G}_h$.\\
Using these spaces by Theorem \ref{d4} and Proposition \eqref{p4.6} we deduce the following result
\begin{proposition}
For every $k\in \N$ the curve of radial solution $(\l,u_{\l})\in (-\infty,0)\times X$ contains a nonradial bifurcation point in the interval $I_k\times \mathcal{H}$, where $I_k$ is as defined in Proposition \eqref{p4.6}.\\
Moreover if $k$ is even, for every $h=1,\dots, \left[\frac N2\right]$ 
there exists a continuum  of nonradial solution bifurcating from $(\l,u_{\l})$ in the interval $I_k\times \mathcal{H}^h$.
\end{proposition}
\begin{proof}
The proof is by contradiction. We consider only the case of the space $\mathcal{H}$. The other case follows in a very similar way.\\  
Assume by contradiction that the curve $(\l,u_{\l})$ does not contain any bifurcation point in the interval $I_k\times \mathcal{H}$. Then there exists an $\e_0>0$ such that for
$\e\in (0,\e_0)$ and every $c\in (0,\e_0)$ we have
\begin{equation}\label{2.55}
v-T(\l,v)\neq 0,\quad \forall \l\in I_k, \forall v\in X\hbox{
  such that }0<\nor v-u_{\l}\nor_{X}\leq c.
\end{equation}
Let us consider the
set  $\mathcal{C}:=\{(\l,v)\in I_k\times  X\,:\, \nor v-u_{\l}\nor_X
<c\}$ and $ \mathcal{C}_{\l}:=\{v\in X\hbox{ such that } (\l,v)\in \mathcal{C} \}$.
From
(\ref{2.55}) it follows that there  exist no  solutions of $v-T(\l,v)=0$ on $\de
_{I_k\times X}\mathcal{C}$ different from the radial ones. By the homotopy invariance of
the degree, we get
\begin{equation}\label{2.66}
\mathit{deg} \left( I-T(\l,\cdot), \mathcal{C}_\l,0\right)\hbox{ is constant on }I_k.
\end{equation}
Moreover from \eqref{d5}, \eqref{f0a},  \eqref{f0b},  \eqref{f0c}, and  \eqref{f0c*} we have that the linearized  operator
$T'(\l,u_{\l})$ is invertible for  $\l=\a_k$ and
$\l=\b_k$. Then  
$$\mathit{deg} \left(I- T(\b_k ,\cdot),  \mathcal{C} _{ \b_k  },0\right)=(-1)^{m_{\mathcal{H}}(\b_k )}$$ 
and
$$\mathit{deg} \left(I- T(\a_k,\cdot),  \mathcal{C} _{ \a_k } ,0\right)=(-1)^{m_{\mathcal{H}}(  \a_k )}$$
where $m_{\mathcal{H}}(\l)$ denotes the Morse index of the radial solution $u_\l$ in the space $\mathcal{H}$. By the choice of the space $\mathcal{H}$ we know that the eigenspace of the Laplace Beltrami operator associated to $\mu_k$ is one-dimensional. Then, repeating the proof of Corollary \ref{c4.8} in the space $\mathcal{H}$ we have that 
\begin{equation}\nonumber
m_{\mathcal{H}}(\l)=
\left\{\begin{array}{ll}
1+\sup\{j\in \N\ \text{s.t. }\L(\l)+b^2\mu_j<0\} & \text{if } \l \text{ does not satisfy \eqref{d5}}\\
\\
\sup\{ j\in \N\ \text{s.t. }\L(\l)+b^2\mu_j<0\} & \text{if } \l \text{ satisfies \eqref{d5}}
\end{array}\right.
\end{equation}
Then, from \eqref{f0a}-\eqref{f0c*} we have that $m_{\mathcal{H}}(\b_k )=1+(k-1)=k$ and $m_{\mathcal{H}}(  \a_k )=1+k$
%from \eqref{gamma_i}, by the choice of the space $\mathcal{H}$ and by the choice of the points $\b_k$ and $\a_k$ that satisfy the properties \eqref{f0a}-\eqref{f0c*}, we have 
so that 
$$\mathit{deg} \left(I- T(\b_k,\cdot),  \mathcal{C} _{ \b_k},0\right)=-\mathit{deg} \left(I- T(\a_k,\cdot),  \mathcal{C} _{ \a_k} ,0\right)$$
contradicting \eqref{2.66}. Then, in the interval $I_k\times X$ there exists a bifurcation point for the curve $(\l,u_{\l})$ and the bifurcating solutions are nonradial since $u_\l$ is radially nondegenerate.
\end{proof}

%\begin{theorem}
% Let $\l_k$ as in the previous theorem. Then\\
%i) If $k$ is odd there exists at least a continuum of nonradial solutions to 
%\begin{equation}\label{00t}
%\left\{\begin{array}{ll}
%-\Delta u-\frac{\l}{|x|^2}u=u^p& \hbox{ in }B_1\\
%u>0&\hbox{ in }B_1\\
%u=0&\hbox{ on }\partial B_1.
%\end{array}\right.
%\end{equation}
%for  $1<p<\frac{N+2}{N-2}$, invariant with respect to $O(N-1)$, bifurcating from the pair $(\l_k,u_{\l_k})$.\\
%ii) If $k$ is even there exist at least $\big[\frac N2\big]$ continua of nonradial solutions to
%\eqref{00t} bifurcating from the pair $(\l_k,u_{\l_k})$. The first branch is $O(N-1)$ invariant, the second is $O(N-2)\times O(2)$ invariant, etc.
%\end{theorem}
%\begin{proof}
%The proof is basically the same of the previous section.
%\end{proof}

\sezione{Appendix}
\begin{lemma}\label{l2}
Let $\l\in(-\infty,\frac{(N-2)^2}4)$. Then 
\begin{equation}\label{norm-equivalent}
\left(\int_{\R^N}|\na v|^2\, dx-\int_{\R^N}\frac {\l}{|x|^2}v^2\, dx\right)^{\frac 12}
\end{equation}
is a norm on $D^{1,2}(\R^N)$ which is equivalent to the standard one.
\end{lemma}
\begin{proof}
It follows by the Hardy inequality distinguishing the two different cases, $\l> 0$ and $\l\leq 0$.
%We have to distinguish the two different cases, $\l> 0$ and $\l\leq 0$.
%Consider, first, the case $\l> 0$.
%Using the Hardy inequality we have that, for any $g\in D^{1,2}(\R^N)$ it holds
%$$\left(1-\frac{4\l}{(N-2)^2}\right)\int_{\R^N}|\na g|^2\,dx\leq \int_{\R^N}|\na g|^2\,dx-\int_{\R^N}\frac {\l}{|x|^2}g^2\, dx\leq\int_{\R^N}|\na g|^2\,dx$$
%and, since  $1-\frac{4\l}{(N-2)^2}=c>0$, the thesis follows by assumption.\\
%If, else, $\l<0$ then, using again the Hardy inequality we have 
%$$\int_{\R^N}|\na g|^2\,dx\leq \int_{\R^N}|\na g|^2\,dx-\int_{\R^N}\frac {\l}{|x|^2}g^2\, dx\leq \left(1-\frac{4\l}{(N-2)^2}\right)\int_{\R^N}|\na g|^2\,dx$$
%and the proof is complete.
\end{proof}
\begin{lemma}\label{l1}
Let $f(x)\in L^{\frac{2N}{N+2}}(\R^N)$ and let $\l\in(-\infty,\frac{(N-2)^2}4)$. Then the equation
\begin{equation}\label{eq-lineare}
-\Delta v- \frac{\l}{|x|^2}v=f\quad \hbox{ in }\R^N
\end{equation}
has a unique weak solution in $D^{1,2}(\R^N)$.
\end{lemma}
\begin{proof}
It follows by the Hardy inequality and the coercivity of the functional
$$J(v):=\frac 12 \int_{\R^N}|\na v|^2\, dx-\frac 12\int_{\R^N}\frac {\l}{|x|^2}v^2\, dx-\int_{\R^N}fv\, dx.$$
\end{proof}
Next we state the Pohozaev identity for a weak solution of
\begin{equation}\label{appendix}
-\Delta u=f(x,u) \quad \hbox{ in }\R^N
\end{equation}
\begin{lemma}
Let $u\in D^{1,2}(\R^N)$ be a weak solution of \eqref{appendix} and let $F(x,u)=\int_0^u f(x,t)\, dt$. Assume furthermore that $u\in L^{\infty}_{loc} (\R^N\setminus \{0\})$ and that $F(x,u), x\cdot F_x(x,u) \in L^1(\R^N)$, where $F_x$ is the gradient of $F$ with respect to $x$. Then $u$ satisfies
\begin{equation}\label{pohozaev}
\int_{\R^N}|\na u|^2-\frac {2N}{N-2}\int_{\R^N}F(x,u)-\frac 2{N-2}\int_{\R^N}x\cdot F_x(x,u)=0.
\end{equation}
\end{lemma}
\begin{proof}
We can proceed exactly as in the proof of Proposition 1 of \cite{BL}. There are only two differences: one is the presence of the term $x\cdot F_x(x,u)$ and the  second one is that the solution $u\in L^{\infty}_{loc} (\R^N\setminus \{0\})$ and so we have to integrate \eqref{appendix} in $B_R\setminus B_{\rho}$. Anyway these terms can be handled exactly as in the proof of \cite{BL}.
\end{proof}
Here we prove some results that deal with the infimum \eqref{i12} and some other related results in the same spirit of what we proved in the Section 2 of \cite{GGN2}. 
\begin{proposition}\label{p12}
Let $\Omega\subset\R^N$ be a bounded domain with $0\in\Omega$. Moreover assume that 
\begin{equation}\label{d13}
\nu_1=\inf_{\eta\in H^1_0(\Om)\atop \eta\not\equiv0
}\frac{\int_\Omega|\nabla\eta|^2-\int_\Omega a(x)\eta^2}{\int_\Omega \frac{\eta^2}{|x|^2}}<0
\end{equation}
with $a(x)\in L^\infty(\Omega)$. Then $\nu_1$ is achieved. The function $\psi_1 \in H^1_0(\Omega)$ that achieves $\nu_1$ is  strictly positive in $\Omega\setminus\{0\}$  satisfies
\begin{equation}\label{AAAA}
\int_{\Omega}\nabla \psi_1\cdot \nabla \phi-a(x)\psi_1\phi\, dx=\nu_1\int_{\Omega}\frac{\psi_1\phi}{|x|^2}\, dx 
\end{equation}
for any $\phi\in H^1_0(\Omega)$, 
and the eigenvalue $\nu_1$ is simple.
\end{proposition}
\begin{proof}
Let us consider a minimizing sequence $\eta_n\in H^1_0(\Om)$ for $\nu_1$, i.e.,
\begin{equation}\label{v1}
\frac{\int_\Omega|\nabla\eta_n|^2-\int_\Omega a(x)\eta_n^2}{\int_\Omega \frac{\eta_n^2}{|x|^2}}=\nu_1+o(1).
\end{equation}
Let us normalize $\eta_n$ such that 
\begin{equation}\label{v2}
\int_\Om\eta_n^2=1.
\end{equation}
 Then, since $\nu_1<0$,  by \eqref{v1} we get
\begin{equation}\label{3}
\int_\Omega|\nabla\eta_n|^2-\int_\Omega a(x)\eta_n^2\le0
\end{equation}
and then, since $a$ is bounded and \eqref{v2} we deduce from \eqref{3} that
\begin{equation}\label{4}
\int_\Omega|\nabla\eta_n|^2\le C\int_\Omega\eta_n^2\le C.
\end{equation}
Hence $\eta_n\rightharpoonup\eta$ weakly in $H^1_0(\Om)$ and then it holds,
\begin{equation}\label{5a}
\int_\Omega|\nabla\eta|^2\le\liminf_{n\rightarrow+\infty}\int_\Omega|\nabla\eta_n|^2
\end{equation}
\begin{equation}\label{5b}
\int_\Omega a(x)\eta_n^2\rightarrow\int_\Omega a(x)\eta^2.
\end{equation}
So we get
\begin{equation}
\int_\Omega|\nabla\eta|^2-\int_\Omega a(x)\eta^2\le\liminf_{n\rightarrow+\infty}\int_\Omega|\nabla\eta_n|^2-\int_\Omega a(x)\eta_n^2+o(1)
\end{equation}
which implies, since $\nu_1<0$, and $1=\int_\Omega\eta_n^2\le C\int_\Omega \frac{\eta_n^2}{|x|^2}$,
\begin{equation}\label{5}
\frac{\int_\Omega|\nabla\eta|^2-\int_\Omega a(x)\eta^2}{\limsup\limits_{n\rightarrow+\infty}\int_\Omega \frac{\eta_n^2}{|x|^2}}
%\le\frac{\int_\Omega|\nabla\eta_n|^2-\int_\Omega a(x)\eta_n^2+o(1)}{\int_\Omega \frac{\eta_n^2}{|x|^2}}=
\le\nu_1.
\end{equation}
Then elementary properties of lim inf and lim sup imply
\begin{equation}\label{55}
\limsup\limits_{n\rightarrow+\infty}\frac{\int_\Omega|\nabla\eta|^2-\int_\Omega a(x)\eta^2}{\int_\Omega \frac{\eta_n^2}{|x|^2}}\le\nu_1.
\end{equation}
Moreover, by Hardy's inequality
\begin{equation}
\int_\Omega \frac{\eta_n^2}{|x|^2}\le \frac{(N-2)^2}4\int_\Omega|\nabla\eta_n|^2\le C
\end{equation}
and so, by semicontinuity,
\begin{equation}\label{6a}
\int_\Omega\frac{\eta^2}{|x|^2}\le \liminf\limits_{n\rightarrow+\infty}\int_\Omega \frac{\eta_n^2}{|x|^2}.
\end{equation}
Hence, again using that $\nu_1<0$, we get from \eqref{5} that
\begin{equation}\label{6}
\int_\Omega|\nabla\eta|^2-\int_\Omega a(x)\eta^2<0.
\end{equation}
On the other hand, from \eqref{6a} we get
\begin{equation}\label{7}
\limsup\limits_{n\rightarrow+\infty}\frac1{\int_\Omega \frac{\eta_n^2}{|x|^2}}= \frac1{\liminf\limits_{n\rightarrow+\infty}\int_\Omega \frac{\eta_n^2}{|x|^2}} \le\frac1{\int_\Omega\frac{\eta^2}{|x|^2}}
\end{equation}
and then
\begin{equation}\label{8}
\frac{\int_\Omega|\nabla\eta|^2-\int_\Omega a(x)\eta^2}{\int_\Omega \frac{\eta^2}{|x|^2}}\le\liminf\limits_{n\rightarrow+\infty}\frac{\int_\Omega|\nabla\eta|^2-\int_\Omega a(x)\eta^2}{\int_\Omega \frac{\eta_n^2}{|x|^2}}.
\end{equation}
Finally by \eqref{5} we get
\begin{equation}
\frac{\int_\Omega|\nabla\eta|^2-\int_\Omega a(x)\eta^2}{\int_\Omega \frac{\eta^2}{|x|^2}}\le\nu_1
\end{equation}
which proves the first part of the proposition. The rest follows exactly as in proof of Proposition 2.1 of \cite{GGN2}
\end{proof}
The same result holds also if we minimize the quadratic form \eqref{d13} with some orthogonality conditions. To this end we say that $\psi$ and $\eta$ are orthogonal if they satisfy $\int_{\Omega}\frac{\psi \eta}{|x|^2}\, dx=0$. Indeed we have the following:
\begin{proposition}\label{d12-altri}
Let us assume $\Om$, $\nu_1$, $\psi_1$ and $a(x)$ as in  Proposition \ref{p12}. 
Then if we have that
\begin{equation}\label{9}
\nu_2=\inf_{\scriptstyle
\eta\in H^1_0(\Omega) ,\atop\scriptstyle
\eta\perp\psi_1}\frac{\int_\Omega|\nabla\eta|^2-\int_\Omega a(x)\eta^2}{\int_\Omega \frac{\eta^2}{|x|^2}}<0
\end{equation}
then $\nu_2$ is achieved. Moreover the functions $\psi_2\in H^1_0(\Omega)$ that attains $\nu_2$ satisfies
$$
\int_{\Omega}\nabla \psi_2\cdot \nabla 
\phi-a(x)\psi_2\phi\, dx=
\nu_2\int_{\Omega}\frac{\psi_2\phi}{|x|^2}\, dx $$
for any $\phi\in H^1_0(\Omega)$.\par
Similarly for $i=3,..,k$, if we have that
\begin{equation}
\nu_i=\inf_{\scriptstyle
\eta\in H^1_0(\Omega)  ,\atop\scriptstyle
\eta\perp span\left\{\psi_1,\psi_2,..,\psi_{i-1}\right\}}\frac{\int_\Omega|\nabla\eta|^2-\int_\Omega a(x)\eta^2}{\int_\Omega \frac{\eta^2}{|x|^2}}<0
\end{equation}
then $\nu_i$ is achieved and the functions $\psi_i\in H^1_0(\Omega)$ that attain $\nu_i$ satisfy
$$
\int_{\Omega}\nabla \psi_i\cdot \nabla 
\phi-a(x)\psi_i\phi\, dx=
\nu_i\int_{\Omega}\frac{\psi_i\phi}{|x|^2}\, dx $$
for any $\phi\in H^1_0(\Omega)$.
\end{proposition}
\begin{proof}
It is the same of the previous lemma. For any $i$ let us consider a
minimizing sequence $\eta_{i,n}\in H^1_0(\Omega)$  for $\nu_i$ . Then it converges to a function $\eta_i$
which achieves $\nu_i$ and that is a weak solution of the equation.
\end{proof}
Now we use the previous result to compute the Morse index of the radial solution $u_\l$ to \eqref{first}. We state the result in the case of $\Omega=B_1$. 
\begin{lemma}\label{B100}
Let $u_{\l}$ be a solution to \eqref{i10} whose Morse index is $M>0$. Then there exist exactly $M$ functions $\psi_i\in H^1_0(B_1)$ and $M$ numbers $\nu_i<0$ such that the problem
\begin{equation}\label{fin-fine}
\begin{cases}
-\Delta \psi_i-\frac{\l}{|x|^2}\psi_i-pu_{\l}^{p-1}\psi_i=\frac{\nu_i}{|x|^2}\psi_i, \ \text{in} \ \ B_1\setminus\{0\}\\
\psi_i\in H^1_0(B_1)
\end{cases}
\end{equation}
admits a weak solution. The functions $\psi_i$ can be taken in such a way they verify
\begin{equation}\label{perpendicolari}
\int_{B_1}\frac{ \psi_i\psi_j}{|x|^2}\, dx=0 \quad \text{ for }i\neq j.
\end{equation}
\end{lemma}
The proof follows exactly as in the proof of Lemma 2.6 in \cite{GGN2} and we do not report it.\\
The results of Propositions \ref{p12} and \ref{d12-altri} and Lemma \ref{B100} hold true also if we let $\Omega=\R^N$ and substitute $H^1_0(\Omega)$ with $D^{1,2}(\R^N)$. Then we can state the following:
\begin{corollary}\label{d12-mi}
The Morse index of the radial solution $u_\l$ of \eqref{1} 
is given by the number of negative values $\L_i$ such that the problem
\begin{equation}\label{ciclamino}
\left\{\begin{array}{ll}
-\Delta w-\frac {\l}{|x|^2}w-N(N+2)\nu_\l^2 \frac{|x|^{2 \nu_\l}} {\left(1+ |x|^{2\nu_\l}\right)^ 2}     w=\frac {\L_i}{|x|^2}w  & \hbox{ in }\R^N\\
w\in D^{1,2}\left(\R^N\right)
\end{array}\right.
\end{equation}
admits a weak solution, counted with their multiplicity.
\end{corollary}
\begin{proof}
Let $u_\l$ be the radial solution of \eqref{1}. %Then its Morse index is finite for every $\l<\frac{(N-2)^2}4$. 
Then we can use the analouguos of Lemma \ref{B100} in $\R^N$ to prove the claim.
\end{proof}
Finally the result of Proposition \ref{p12} can be used also to prove that the first eigenvalue with weight \eqref{i12} is attained. Indeed we have:
\begin{proposition}\label{d12-bis}
Assume that 
\begin{equation}\label{d19}
\Lambda_1=\inf_{\eta\in H^1((0,1),r^{M-1}dr), \eta(1)=0\atop\eta\not\equiv0
}\frac{\int_0^1 r^{M-1}|\eta'|^2\,dr-\int_0^1r^{M-1} a(r)\eta^2 \,dr}{\int_0^1 r^{M-3} \eta^2\, dr}<0.
\end{equation}
with $a\in L^\infty(0,1)$. Then $\L_1$ is achieved.
\end{proposition}
\begin{proof}
The claim follows as in the proof of Proposition \ref{p12}.
\end{proof}

\begin{lemma}\label{B2}
Let us consider a solution to
\begin{equation}\label{f1a}
\begin{cases}
-\psi'' - \frac{M-1}{r}\psi' +\beta^2\frac{\psi}{r^2}=h\psi, \quad \text{in} \ \ (0,1)\\
\psi(1)=0,\ \int_0^1 r^{M-1}(\psi')^2dr < \infty
\end{cases}
\end{equation}
with $h\in L^\infty(0,1)$ and $\beta\ne0$. Then $\psi\in L^\infty(0,1)$ and $\psi(0)=0$.
\end{lemma}
\begin{proof}
Let $\theta=\frac{2-M+\sqrt{(M-2)^2+4\beta^2}}2>0$.
Since $\int_0^1 r^{M-1}(\psi')^2dr<+\infty$ we get by \eqref{f1a} that $\int_0^1\psi^2r^{M-3}dr<+\infty$. Then there exists a sequence $r_n\rightarrow0$ such that
$r_n^{\theta+M-2}\psi(r_n)=o(1)$ as $n\rightarrow+\infty$ for any $\beta>0$. Such a sequence exists because, if not, we get $\psi(r)\ge\frac C{r_n^{\theta+M-2}}$ in a suitable neighborhood of $0$ and this contradicts that $\int_0^1\psi^2r^{M-3}dr<+\infty$ (note that we have used that $\theta>\frac{2-M}2$).\\
Let us observe that the function $v(r)=r^{\theta}$ satisfies
 \begin{equation}\label{z2-bis}
\begin{cases}
-v''-\frac{M-1}r v'+\frac{\beta^2}{r^2}v=0, \quad\hbox{in }(0,+\infty)\\
v(0)=0.
\end{cases}
\end{equation}
 From \eqref{z2-bis} and \eqref{f1a} we obtain, integrating on $(r_n,R)$,
 \begin{equation}\label{f2}
 \int_{r_n}^Rs^{\theta+M-1}h(s)\psi(s)ds=-R^{\theta+M-1}\psi'(R)+r_n^{\theta+M-1}\psi'(r_n)+\theta R^{\theta+M-2}\psi(R)-\theta r_n^{\theta+M-2}\psi(r_n)
\end{equation}
We claim that
\begin{equation}\label{f3}
r_n^{\theta+M-1}\psi'(r_n)=o(1)\ , \ \hbox{ as }n\to \infty
\end{equation}
Integrating \eqref{f1a} we get
\begin{equation*}
r_n^{\theta+M-1}\psi'(r_n)=O\left(r_n^{\theta}\right)+r_n^{\theta}
\int_{r_n}^1 s^{M-1}h(s)\psi(s)ds
-\beta^2r_n^{\theta}\int_{r_n}^1s^{M-3}\psi(s)ds=o(1)
\end{equation*}
since $r_n^{\theta}\int_{r_n}^1s^{M-3}\psi(s)\le r_n^{\theta} \left(\int_{r_n}^1\frac{\psi^2(s)}{s^2}s^{M-1}\right)^\frac12
\left(\int_{r_n}^1s^{M-3}\right)^\frac12=o(1)$ %by Hardy's inequality
and this proves  \eqref{f3}. Hence  \eqref{f2} becomes
 \begin{equation}\label{f4}
 \int_{0}^Rs^{\theta+M-1}h(s)\psi(s)ds=-R^{\theta+M-1}\psi'(R)+\theta R^{\theta+M-2}\psi(R)
\end{equation}
Then we deduce that
 \begin{equation}\label{f5}
\frac{\psi(t)}{t^{\theta}}=\int_t^1\frac1{R^{2\theta+M-1}}
\left(\int_0^Rs^{\theta+M-1}h(s)\psi(s)ds\right)dR.
\end{equation}
Now, since $h\in L^\infty(0,1)$ we get
\begin{align}\label{B.12}
\left|\int_0^Rs^{\theta+M-1}h(s)\psi(s)\right|\le C\int_0^Rs^{\theta+\frac{M+1}2}\psi(s)s^\frac{M-3}2 & \le C \left(\int_0^Rs^{2\theta+M+1}\right)^\frac12
\left(\int_0^R\psi^2(s)s^{M-3}\right)^\frac12\nonumber \\
&\leq C R^{\theta+\frac{M+2}2}.
\end{align}
Finally \eqref{f5} becomes
\begin{equation}\label{f6}
\psi(t)=\begin{cases}O\left(t^{\theta}\right)&\hbox{ if }\theta<-\frac{M}2+3\\
O\left(t^{-\frac{M}2+3}\right)&\hbox{ if }\theta>-\frac{M}2+3\\
O\left(t^{-\frac{M}2+3}|\log t|\right)&\hbox{ if }\theta=-\frac{M}2+3
\end{cases}
\end{equation}
Then if $M<6$ \eqref{f6} implies that $\psi(0)=0$ which gives the claim.
If $M\ge6$, instead, we have $-\frac{M}2+3\leq0<\theta$ and so $\psi(t)=O\left(t^{-\frac{M}2+3}\right)$. Plugging this estimate in \eqref{B.12} we get
\begin{equation}\label{f7}
\left|\int_0^Rs^{\theta+M-1}h(s)\psi(s)\right|\leq C R^{\theta+\frac{M+6}2}.
\end{equation}
and \eqref{f5} becomes
\begin{equation}\label{f8}
\psi(t)=\begin{cases}O\left(t^{\theta}\right)&\hbox{ if }\theta<-\frac{M}2+5\\
O\left(t^{-\frac{M}2+5}\right)&\hbox{ if }\theta\ne-\frac{M}2+5\\
O\left(t^{-\frac{M}2+5}|\log t|\right)&\hbox{ if }\theta=-\frac{M}2+5
\end{cases}
\end{equation}
which gives the claim for $M<10$. Iterating the procedure in a finite number of steps we get that $\psi(t)=O\left(t^{\theta}\right)$ as $t\to 0$ so that $\psi(0)=0$. This ends the proof.
\end{proof}


\begin{thebibliography}{99}
\bibitem[AM]{AM}{\sc A. Ambrosetti, A. Malchiodi},
Nonlinear analysis and semilinear elliptic problems,
Cambridge Studies in Advanced Mathematics, 104, Cambridge University Press, Cambridge, 2007.

\bibitem[BL]{BL}
	{\sc H. Berestycki, P.L. Lions}, Nonlinear scalar field equations, I. Existence of a ground state,
	{\em Arch. Ration. Mech. Anal.} {\bf 82} (1983), 313--346.

\bibitem[CGS]{CGS}{\sc L. Caffarelli, B. Gidas, J. Spruck}, Asymptotic symmetry and local behavior of semilinear elliptic equations with critical Sobolev growth, {\em Comm. Pure Appl. Math.}, {\bf 42}, (1989), 271-297.

\bibitem[CG]{CG}{\sc M. Chaves, J.  Garcia-Azorero}, On uniqueness of positive solutions of semilinear elliptic equations with singular potential, {\em Adv. Nonlinear Stud.} ,{\bf 3}, (2003), 273-288.

\bibitem[D]{D}{\sc E.N. Dancer}, Real analyticity and non-degeneracy, {\em Math. Ann.}, {\bf 325}, (2003), 369-392


  
\bibitem[GP]{GP} {\sc J. P. Garcia Azorero, I. Peral}, Hardy Inequalities and Some Critical Elliptic and
Parabolic Problems, {\em  Journ. Diff. Eqns}, {\bf 144} (1998), 441-476.

\bibitem[G]{G} {\sc F. Gladiali}, A global bifurcation result for a
  semilinear elliptic equation, {\em  Journ. Math. Anal.  Appl.} {\bf 369} (2010), 306-311.

\bibitem[GGN]{GGN} {\sc F. Gladiali, M. Grossi, S. Neves}, Nonradial solutions for the H\'enon equation in $R^N$,  Adv. Math. {\bf 249} (2013), 1-36.

\bibitem[GGN2]{GGN2} {\sc F. Gladiali, M. Grossi, S. Neves}, Symmetry breaking and Morse index of solutions of nonlinear elliptic problems in the plane, arXiv:1308.0519 

%\bibitem[GGPS]{GGPS}  {\sc F. Gladiali, M. Grossi, F. Pacella, P.N. Srikanth},  Bifurcation and symmetry breaking for a class of semilinear elliptic equations in an annulus, Calc. Var. Partial Differential Equations {\bf 40} (2011),  295-317.

\bibitem[GNN]{GNN} {\sc B. Gidas, W. M. Ni and L. Nirenberg},  Symmetry and related properties via the maximum principle, Comm. Math. Phys. 68 (1979), 209-243.

\bibitem[JLX]{JLX} {\sc Q. Jin, Y. Li and H. Xu} Symmetry and asymmetry: the method of moving spheres. Adv. Differential Equations 13 (2008),  601–640.

\bibitem[K]{K} {\sc T. Kato}, Perturbation Theory for Linear Operators, Springer-Verlag, Berlin, 1976.

\bibitem[MW]{MW} {\sc M. Musso and J. Wei} Nonradial solutions to critical elliptic equations of Caffarelli-Kohn-Nirenberg type. Int. Math. Res. Not. IMRN 2012,  4120–4162. 

\bibitem[Ni]{Ni} {\sc W. M. Ni}, A Nonlinear Dirichlet Problem on the Unit Ball and Its Applications,
 {\em Indiana Univ. Math. J.} {\bf 31}, (1982), 801-807.
 
\bibitem[S]{S} {\sc P. N. Srikanth}, Uniqueness of solutions of nonlinear Dirichlet problems, Diff. Int. Eqns 6 (1993), 663�670.

\bibitem[SW86]{SW86} {\sc J. Smoller, A. Wasserman}, Symmetry-breaking for solutions of semilinear elliptic equations with general boundary conditions, {\em Comm. Math. Phys. } {\bf 105}, (1986), 415-441.

\bibitem[SW90]{SW90} {\sc J. Smoller, A. Wasserman}, Bifurcation and symmetry-breaking, {\em Invent.
Math.} {\bf 100},  (1990), 63-95.


\bibitem[Ta]{Ta} {\sc G. Talenti}  Best constant in Sobolev inequality. {\em Ann Mat Pura Appl.} {\bf 110}, (1976), 353-372.


\bibitem[T]{T}{\sc S. Terracini}, On positive entire solutions to a class of equations with a singular coefficient and critical exponent, Adv. Diff. Eq. 1,(1996), 241-264.
\end{thebibliography}
\end{document}